\documentclass[12pt,thmsa]{article}
\usepackage{amssymb}
\usepackage{amsfonts}
\usepackage{enumerate}
\usepackage{amsmath}
\usepackage{cite}
\usepackage[top=3.8cm,bottom=3.8cm,textwidth=16cm,centering,a4paper]{geometry}
\usepackage{tikz}
\usepackage{caption2}
\usepackage{subfigure}
\usepackage{tcolorbox}

\newtheorem{theorem}{Theorem}[section]

\newtheorem{proposition}[theorem]{Proposition}
\newtheorem{corollary}[theorem]{Corollary}
\newtheorem{remark}[theorem]{Remark}

\newtheorem{lemma}[theorem]{Lemma}

\newtheorem{definition}[theorem]{Definition}
\newenvironment{proof}[1][Proof]{\noindent\textbf{#1.} }{\ \rule{0.5em}{0.5em}}

\begin{document}

\title{Standing waves with prescribed mass for Schr\"{o}dinger equations with competing Van Der Waals type potentials}
\date{}
\author{Shuai Yao$^{a}$\thanks{%
E-mail address: shyao2019@163.com (S. Yao)}, Hichem Hajaiej$^{b}$\thanks{%
E-mail address: hhajaie@calstatela.edu (H. Hajaiej)}, Juntao Sun$^{a}$%
\thanks{%
E-mail address: jtsun@sdut.edu.cn (J. Sun)} \\
{\footnotesize $^{a}$\emph{School of Mathematics and Statistics, Shandong
University of Technology, Zibo 255049, PR China}}\\
{\footnotesize $^{b}$\emph{Department of Mathematics, California State
University at Los Angeles, Los Angeles, CA 90032, USA}}}
\maketitle

\begin{abstract}
We investigate standing waves with prescribed mass for a class of Schr\"{o}dinger equations with competing Van Der Waals type
potentials, arising in a model of non-relativistic bosonic atoms and molecules. By developing an approach based on a direct minimization of the
energy functional on a new constrained manifold, we establish the existence of two normalized solutions for the corresponding stationary problem. One is a local minimizer with positive level and the other one is a global minimizer with negative level. Moreover, we find that the global minimizer is farther away from the origin than the local minimizer. Finally, we explore the relations between the ground state solution and the least action solution, and some dynamical behavior and scattering results are presented as well.
\end{abstract}



\section{Introduction}

In this paper, we are concerned with the following Schr\"{o}dinger equations
with Van Der Waals type potentials
\begin{equation}
i\partial _{t}\psi +\Delta \psi +(W_{\alpha ,\beta }(x)\ast |\psi |^{2})\psi
=0\text{ in }\mathbb{R\times R}^{N},  \label{e1-6}
\end{equation}%
where $\psi (t,x)$ is the complex valued function in the spacetime $\mathbb{%
R\times R}^{N}$ ($N\geq 3$), and
\begin{equation*}
W_{\alpha ,\beta }(x)=\mu _{\alpha }|x|^{-\alpha }+\mu _{\beta }|x|^{-\beta
}\quad \text{with }\mu _{\alpha },\mu _{\beta }\in \mathbb{R}.
\end{equation*}%
(\ref{e1-6}) comes from a model of non-relativistic bosonic
atoms and molecules, which has two interacting potentials, and this
interaction is weaker and has a longer range than the Dirac delta type
potential, see \cite{FL}. Physically, if $\alpha >1$ and $\beta >1$, $%
W_{\alpha ,\beta }(x)$ is the Van Der Waals type potential, also known as
the Van Der Waals interaction, it represents the potential energy between two
neutral molecules and is a consequence of the attractive and repulsive
forces between the molecules. It is a crucial concept in the study of
intermolecular forces in physics and chemistry, see \cite{DLP,YLT,ZN}. The
Van Der Waals coefficient $\mu _{6}$, $\mu _{8}$ and $\mu _{10}$ associated
with alkaline-earth interactions, which were calculated by Porsev and
Derevianko \cite{PD} through the employment of relativistic many-body
perturbation theory, are regarded as having an accuracy reaching up to $%
1/100 $. Moreover, (\ref{e1-6}) also arises as a mean-field limit of a
bosonic system with attractive two-body interactions which can be taken
rigorously in many cases \cite{LNR}.

Without loss of generality, we may assume that $\mu _{\alpha
}=1$. According to the distributive law of convolution, this leads to the
study of the following nonlocal equation
\begin{equation}
i\partial _{t}\psi +\Delta \psi +(|x|^{-\alpha }\ast |\psi |^{2})\psi +\mu
_{\beta }(|x|^{-\beta }\ast |\psi |^{2})\psi =0\text{ in }\mathbb{R\times R}%
^{N}.  \label{e1-0}
\end{equation}%
For the Cauchy problem to (\ref{e1-0}), when the initial data $\psi _{0}\in
H^{1}(\mathbb{R}^{N})$, we can study the finite Hamiltonian or the finite
energy solutions, which satisfy the energy (Hamiltonian), mass and momentum
conservations:
\begin{equation*}
E[\psi (t)]:=\frac{1}{2}\int_{\mathbb{R}^{N}}|\nabla \psi |^{2}dx-\frac{1}{4}%
\int_{\mathbb{R}^{N}}(|x|^{-\alpha }\ast |\psi |^{2})|\psi |^{2}dx-\frac{\mu
_{\beta }}{4}\int_{\mathbb{R}^{N}}(|x|^{-\beta }\ast |\psi |^{2})|\psi
|^{2}dx=E[\psi _{0}],
\end{equation*}%
\begin{equation*}
M[\psi (t)]:=\int_{\mathbb{R}^{N}}|\psi |^{2}dx=M[\psi _{0}],
\end{equation*}%
\begin{equation*}
P[\psi (t)]:=\mathrm{Im}\int_{\mathbb{R}^{N}}\bar{\psi}\nabla \psi dx=P[\psi
_{0}].
\end{equation*}%
By using standard contraction mapping argument \cite{BGH,C}, we can get the local
well-posedness result.

\begin{theorem}
\label{t1} $(i)$ (The energy subcritical case) Let $N\geq 3$, $\mu_{\beta}\in \mathbb{R}$, $0<\alpha<\beta <\min \left\{ N,4\right\} $ and
the initial data $\psi (0,x)=\psi _{0}(x)\in H^{1}(\mathbb{R}^{N}).$ Then
there exists $T_{max}>0$ and a unique solution $\psi (t,x)\in C([0,T_{\max
});H^{1})$ of the Cauchy problem to (\ref{e1-0}).\newline
$(ii)$ (The energy critical case) Let $N\geq 5$, $\mu_{\beta}\in \mathbb{R}$, $0<\alpha <\beta=4 $ and
the initial data $\psi (0,x)=\psi _{0}(x)\in H^{1}(\mathbb{R}^{N}).$ If $\|\psi _{0}\|_{H^{1}}\leq\eta$ for some small $\eta>0$, then there exists $T_{max}>0$ and a unique solution $\psi (t,x)\in C([0,T_{\max
});H^{1})$ of the Cauchy problem to (\ref{e1-0}).
\end{theorem}

An important feature related to the nonlinear time dependent equations such
as (\ref{e1-0}) is to study standing waves. A standing wave of (\ref{e1-0})
is a solution of the form%
\begin{equation*}
\psi (t,x)=e^{i\lambda t}u(x),
\end{equation*}%
where the frequency $\lambda \in \mathbb{R}$ and the real value function $u$
satisfies the stationary equation
\begin{equation}
-\Delta u+\lambda u=(|x|^{-\alpha }\ast \left\vert u\right\vert ^{2})u+\mu
_{\beta }(|x|^{-\beta }\ast \left\vert u\right\vert ^{2})u\quad \text{in}%
\quad \mathbb{R}^{N}.  \label{e1-1}
\end{equation}%
Meanwhile, we note that the mass of the solutions to the Cauchy problem
(\ref{e1-0}) is conserved through time, that is, the $L^{2}$-norms of standing
waves are independent of $t$. Therefore it is interesting to study the existence
of standing waves with prescribed mass, which are usually called \textbf{normalized solutions}%
. In this case $\lambda \in \mathbb{R}$ is unknown and appears as a Lagrange
multiplier. This study seems particularly meaningful from the physical point
of view, since this type of solutions enjoys nice stability properties. In order to
find normalized solutions, we need consider the following problem:%
\begin{equation}
\left\{
\begin{array}{ll}
-\Delta u+\lambda u=(|x|^{-\alpha }\ast \left\vert u\right\vert ^{2})u+\mu
_{\beta }(|x|^{-\beta }\ast \left\vert u\right\vert ^{2})u & \quad \text{in}%
\quad \mathbb{R}^{N}, \\
\int_{\mathbb{R}^{N}}|u|^{2}dx=c, &
\end{array}%
\right.  \label{e1-2}
\end{equation}%
where $c>0$ given. Solutions of problem (\ref{e1-2}) can be obtained as
critical points of the energy functional $E:H^{1}(\mathbb{R}^{N})\rightarrow
\mathbb{R}$ given by%
\begin{equation*}
E(u):=\frac{1}{2}A\left( u\right) -\frac{1}{4}B_{\alpha }\left( u\right) -%
\frac{\mu _{\beta }}{4}B_{\beta }\left( u\right)
\end{equation*}%
on the constraint
\begin{equation*}
S(c):=\left\{ u\in H^{1}(\mathbb{R}^{N}):\int_{\mathbb{R}^{N}}|u|^{2}dx=c%
\right\} ,
\end{equation*}%
where
\begin{equation*}
A\left( u\right) :=\int_{\mathbb{R}^{N}}\left\vert \nabla u\right\vert
^{2}dx,\  B_{\alpha }\left( u\right) :=\int_{\mathbb{R}^{N}}\left(
|x|^{-\alpha }\ast \left\vert u\right\vert ^{2}\right) \left\vert
u\right\vert ^{2}dx\text{ and }B_{\beta }\left( u\right) :=\int_{\mathbb{R}%
^{N}}\left( |x|^{-\beta }\ast \left\vert u\right\vert ^{2}\right) \left\vert
u\right\vert ^{2}dx.
\end{equation*}%
It is easy to show that $E$ is a well-defined and $C^{1}$ functional on $%
S(c) $ for $3\leq N\leq 4$, $0<\alpha ,\beta <\min \left\{ N,4\right\}$ and $N\geq 5$, $0<\alpha ,\beta\leq 4$.

When $\alpha =\beta $, (\ref{e1-1}) becomes the following well-known Hartree
equation
\begin{equation}
-\Delta u+\lambda u=\left( |x|^{-\alpha }\ast |u|^{2}\right) u\quad \text{in}%
\quad \mathbb{R}^{N}.  \label{e1-7}
\end{equation}%
Especially, when $N=3$ and $\alpha =1$, (\ref{e1-7}) appears as a model in
quantum theory of a polaron at rest \cite{P}. The time-dependent form of (%
\ref{e1-7}) also describes the self-gravitational collapse of a quantum
mechanical wave-function \cite{P1}. In such case, (\ref{e1-7}) has been well
studied as a mass subcritical issue. For example, Lieb \cite{L5} proved the
existence and uniqueness of normalized solutions by using symmetrization
techniques, and Lions \cite{L} studied the existence and stability issues of
normalized solutions. While for the mass supercritical case $2<\alpha <\min
\left\{ 4,N\right\} $, by a constrained minimization method, Luo \cite{L1}
obtained a normalized solution for (\ref{e1-7}). Recently, for
the Choquard equation
\begin{equation*}
-\Delta u+\lambda u=\left( |x|^{-\alpha }\ast |u|^{p}\right) |u|^{p-2}u\quad
\text{in}\quad \mathbb{R}^{N},
\end{equation*}%
Li and Ye \cite{LY} showed that if $N\geq 3$, $0<\alpha <N$ and $\frac{%
2N-\alpha +2}{N}<p<\frac{2N-\alpha }{N-2}$, then there exists a
mountain-pass normalized solution for each $c>0.$ We refer the reader to \cite{CFHM,H1,HPS,HS1,HS2,HS3,S3,S4,YCRS} and the references therein, which discuss the normalized solution and stability for similar or more general equations.

When $\alpha \neq \beta $, the situation becomes much more complicated. In
order to better investigate the number and properties of normalized
solutions to problem (\ref{e1-2}), we can present the map to help
to understand the geometry of $E|_{S(c)}$. Such map is known as fibering map and was introduced in \cite{S1,S2,T}. Specifically, for each $u\in S(c)$
and $s>0$, we set the dilations
\begin{equation}
u_{s}\left( x\right) :=s^{N/2}u\left( sx\right) \text{ for all }x\in \mathbb{%
R}^{N}.  \label{e1-4}
\end{equation}%
Define the fibering map $s\in \left( 0,\infty \right) \mapsto g_{u}\left(
s\right) :=E\left( u_{s}\right) $ given by
\begin{equation}
g_{u}\left( s\right) =\frac{s^{2}}{2}A\left( u\right) -\frac{s^{\alpha }}{4}%
B_{\alpha }\left( u\right) -\frac{\mu _{\beta }s^{\beta }}{4}B_{\beta
}\left( u\right) .  \label{e1-5}
\end{equation}%
Moreover, we also define the Pohozaev manifold%
\begin{equation}
\mathcal{P}\left( c\right) :=\left\{ u\in S(c):Q\left( u\right) =0\right\}
=\left\{ u\in S(c):g_{u}^{\prime }\left( 1\right) =0\right\} ,  \label{e1-10}
\end{equation}%
where $Q\left( u\right) =0$ is the corresponding Pohozaev type identity of (%
\ref{e1-1}). It is well-known that any critical point of $E|_{S(c)}$ stays
in $\mathcal{P}\left( c\right) $. Notice that the critical points of $g_{u}$
allow to project a function on $\mathcal{P}\left( c\right) .$ Thus,
the monotonicity and the convexity properties of $g_{u}$ strongly affect the
structure of $\mathcal{P}\left( c\right) $ (and in turn the geometry of $%
E|_{S(c)}$), and also have a strong impact on properties of the
time-dependent equation (\ref{e1-0}).

In light of the variances of the $\alpha ,\beta $ and the sign of $\mu
_{\beta }$, the properties of fibering map $g_{u}$ can be summarized into the following
four cases:\newline
\fbox{\textbf{Case}\boldmath{ $(i)$:}} $\mu _{\beta }>0$ and $0<\alpha <\beta \leq2,$ or $\mu _{\beta }<0$
and $0<\beta <\alpha <2$. The fibering map $g_{u}$ has a unique critical
point which is a global minimum, see Figure \ref{F1}. For this case, the
energy functional $E$ is bounded from below on $S(c)$. In \cite{BGH, CJL},
via using the concentration-compactness principle, the normalized solution of problem (\ref{e1-2}) was obtained as the global minimizer
related to the minimization problem on $S(c)$.\newline
\fbox{\textbf{Case}\boldmath{ $(ii)$:}} $\mu _{\beta }>0$ and $0<\alpha <2<\beta \leq \min \left\{
N,4\right\} $. The fibering map $g_{u}$ has exactly two critical points: a
local minimum at $s_{1}=s_{1}(u)$ and a global maximum at $s_{2}=s_{2}(u)$.
Moreover, $g_{u}$ is decreasing in $(0,s_{1})\cup (s_{2},+\infty )$ and
increasing in $(s_{1},s_{2}),$ see Figure \ref{F2}. For this case, the
energy functional $E$ is bounded from below on $\mathcal{P}\left( c\right) $%
, but not on $S(c)$, and $\mathcal{P}\left( c\right) $ is a natural
constraint. In \cite{CJL}, by using the decomposition of $\mathcal{P}\left(
c\right) $ and the min-max principle by Ghoussoub \cite{G1}, two normalized
solutions of problem (\ref{e1-2}) were found for $\beta <\min \left\{
N,4\right\} $ and $N\geq 3$. Later, Jia and Luo \cite{JL} extended main
results in \cite{CJL} to the energy critical case, i.e. $\beta =4$ and $%
N\geq 5.$\newline
\fbox{\textbf{Case}\boldmath{ $(iii)$:}} $\mu _{\beta }>0$ and $2\leq\alpha <\beta <\min \left\{N,4\right\} $%
, or $\mu _{\beta }<0$ and $\max\{\beta,2\} <\alpha <\min \left\{N,4\right\} $. The
fibering map $g_{u}$ has a unique critical point which is a global maximum,
see Figure \ref{F3}. For this case, the energy functional $E$ is bounded
from below on $\mathcal{P}\left( c\right) $, but not on $S(c)$, and $%
\mathcal{P}\left( c\right) $ is a natural constraint. In \cite{BGH,CJL}, via
using the classical Pohozaev manifold method, the existence of normalized
solution of problem (\ref{e1-2}) was proved.\newline
\fbox{\textbf{Case}\boldmath{ $(iv)$:}} $\mu _{\beta }<0$ and $2<\alpha <\beta \leq\min \left\{N,4\right\} $%
. The fibering map $g_{u}$ has exactly two critical points: a local maximum
at $s_{3}=s_{3}(u)$ and a global minimum at $s_{4}=s_{4}(u)$. Moreover, $%
g_{u}$ is increasing in $(0,s_{3})\cup (s_{4},+\infty )$ and decreasing in $%
(s_{3},s_{4}),$ see Figure \ref{F4}. For this case, although the graph of $%
g_{u}$ is clear, it is not easy to prove that the energy functional $E$ is
bounded from below on $S(c)$ because of the complicated competing effect of the
nonlocal terms. Moreover, $\mathcal{P}\left( c\right) $ is no longer a
natural constraint, since we can not guarantee that an arbitrary $u\in S(c)$
can always be projected onto $\mathcal{P}(c)$ for $c>0$. Thus it is totally
different from other cases. So far there are no any results in existing
literature and some interesting problems are open.

\begin{figure*}[t]
\centering%
\subfigure[]{
		\begin{minipage}[b]{0.45\linewidth}
			\centering
			\begin{tikzpicture}[line width=0.8pt,scale=0.6]
				\draw[->] (-1,0) -- (4,0) node[right] {$s$};
				\draw[->] (0,-1.8) -- (0,3.2) node[above] {$g_{u}(s)$};
				\draw[domain=0:2.8, smooth, variable=\x, thick, red] plot ({\x},{\x^2 - 2*\x});
				\node[below left] at (0,0) {$0$};
			\end{tikzpicture}
			\label{F1}
		\end{minipage}} \hfill
\subfigure[]{
		\begin{minipage}[b]{0.45\linewidth}
			\centering
			\begin{tikzpicture}[line width=0.8pt,scale=0.6]
					\draw[->] (-1,0) -- (4,0) node[right] {$s$};
				\draw[->] (0,-2.5) -- (0,2.5) node[above] {$g_{u}(s)$};
				\draw[domain=0:2.8, smooth, variable=\x, thick, red] plot ({\x},{-1.8*sin(142*\x )});
				\node[below left] at (0,0) {$0$};
				\node[right] at (3.5,2.5) { };
			\end{tikzpicture}
			\label{F2}
		\end{minipage}}
\subfigure[]{
	\begin{minipage}[t]{0.45\textwidth}
		\centering
		\begin{tikzpicture}[line width=0.8pt,scale=0.6]
			\draw[->] (-1,0) -- (4,0) node[right] {$s$};
		\draw[->] (0,-2.8) -- (0,2.2) node[above] {$g_{u}(s)$};
		\draw[domain=0:2.8, smooth, variable=\x, thick, red] plot ({\x},{-\x^2 + 2*\x});
		\node[below left] at (0,0) {$0$};
		\node[right] at (3.5,2.5) { };
		\end{tikzpicture}
	\label{F3}
	\end{minipage}} \hfill
\subfigure[]{
	\begin{minipage}[t]{0.45\textwidth}
		\centering
		\begin{tikzpicture}[line width=0.8pt,scale=0.6]
			\draw[->] (-1,0) -- (4,0) node[right] {$s$};
		\draw[->] (0,-2.5) -- (0,2.5) node[above] {$g_{u}(s)$};
		\draw[domain=0:2.8, smooth, variable=\x, thick, red] plot ({\x},{1.8*sin(142*\x )});
		\node[below left] at (0,0) {$0$};
		\node[right] at (3.5,2.5) { };
		\end{tikzpicture}
	\label{F4}
	\end{minipage}}
\caption{Possible forms of fibering maps}
\end{figure*}
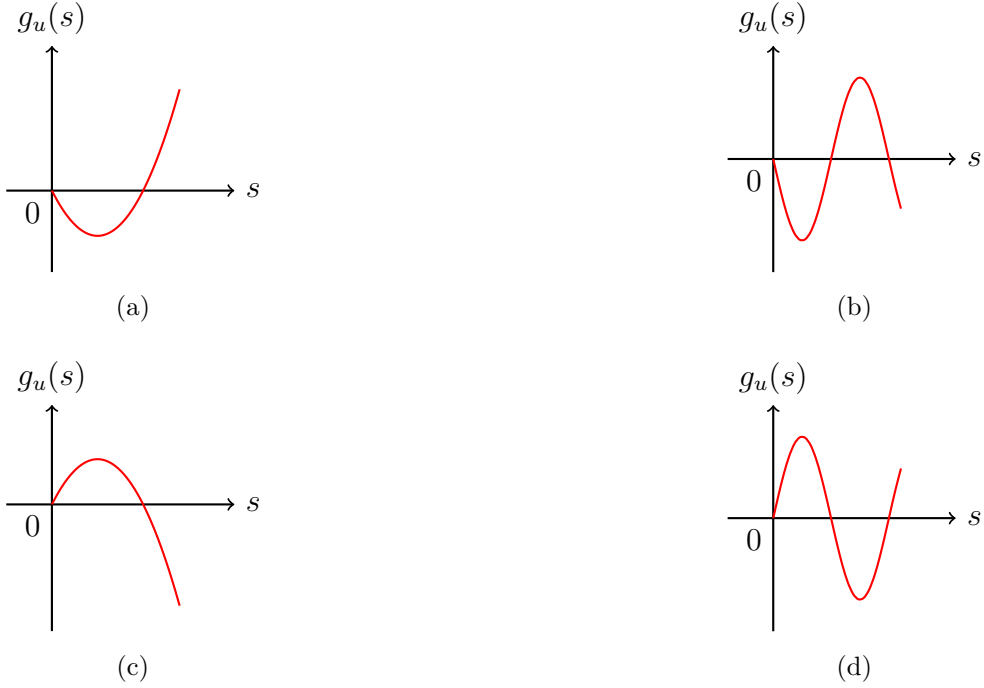

In this paper, we focus on the unsolved case, i.e. Case $(iv)$. Our aim is
to develop a method that will enable us to establish more accurate results on the
multiplicity and dynamics of standing waves with a prescribed mass for (\ref%
{e1-0}). More precisely, we firstly study the existence of two normalized
solutions to (\ref{e1-2}). As a result, the corresponding standing waves
with prescribed mass of (\ref{e1-0}) are obtained. We note that there are
some essential differences in the geometry of the corresponding energy
functionals $E$ between Case $(ii)$ and Case $(iv)$ from the variational
point of view, although they have the same number of normalized solutions.
Secondly, we study the relations between the ground state solution and the least action solution based on a rigorous analysis. It is worth emphasizing that one has
to face more challenges, such as previously mentioned aspects. In order to
overcome considerable difficulties, new ideas and techniques have been
explored. More details will be discussed in the next subsection.

\subsection{Main results}

\begin{definition}
We say that a solution $u\in S(c)$ of equation (\ref{e1-1}) is a ground state if it possesses the minimal energy among all solutions in $S(c)$, i.e., if
\begin{equation*}
E(u)=\inf\left\{E(v):v\in S(c),(E|_{S(c)})^{\prime}(v)=0\right\}.
\end{equation*}
\end{definition}

Let
\begin{equation*}
\mu _{\ast }:=\frac{4(\alpha -2)}{\alpha (\beta -2)\mathcal{S}_{\beta }c^{%
\frac{4-\beta }{2}}}\left( \frac{\beta -\alpha }{\beta -2}\right) ^{\frac{%
\beta -2}{\alpha -2}}\left[ \frac{\beta (\alpha -2)^{2}}{2\alpha ^{2}(\beta
-2)m_{\infty }}\right] ^{\frac{\beta -2}{2}},  \label{e1-8}
\end{equation*}%
where $\mathcal{S}_{\beta }>0$ is as the best constant of Gagliardo-Nirenberg inequality in Lemma \ref{L2-1} below, and $m_{\infty }>0$ is the ground state energy of the following problem:%
\begin{equation*}
\left\{
\begin{array}{ll}
-\Delta u+\lambda u=\left( |x|^{-\alpha }\ast u^{2}\right) u & \text{ in }%
\mathbb{R}^{N}, \\
\int_{\mathbb{R}^{N}}|u|^{2}dx=c>0, &
\end{array}%
\right.\tag{$HE_{\infty }$}
\end{equation*}%
where $N\geq 3$ and $2<\alpha <\min \left\{ N,4\right\}$.

\begin{theorem}
\label{t2} Assume that $\mu_{\beta}<0$ and one of the two following conditions hold:\newline
$(i)$ $N\geq 3$ and $2<\alpha <\beta <\min \left\{ N,4\right\};$\newline
$(ii)$ $N\geq 5$ and $2<\alpha<\beta=4$.\newline
Then for
\begin{equation*}
0<|\mu _{\beta }|<\min \left\{ \mu _{\ast },\frac{\alpha (\beta -2)}{\beta
(\beta -\alpha )}\left(
\frac{2}{\alpha }\right) ^{\frac{\beta -2}{\alpha -2}}\left( \frac{\alpha (\beta -2)}{\beta (\alpha -2)}\right) ^{%
\frac{\beta -2}{2}}\mu _{\ast }\right\} ,
\end{equation*}%
problem (\ref{e1-2}) has two normalized solutions $u^{-},\tilde{u}\in H^{1}(%
\mathbb{R}^{N})$ satisfying
\begin{equation*}
E(\tilde{u})<0<E(u^{-})<\frac{\alpha (\beta -2)^{2}m_{\infty }}{\beta
^{2}(\alpha -2)}\left[ \frac{\alpha (\beta -2)}{\beta (\beta -\alpha )}%
\right] ^{2/(\beta -2)}
\end{equation*}%
and
\begin{equation*}
\Vert \nabla u^{-}\Vert _{2}^{2}<\frac{2\alpha ^{2}(\beta -2)m_{\infty }}{%
\beta (\alpha -2)^{2}}\left[ \frac{\alpha (\beta -2)}{\beta (\beta -\alpha )}%
\right] ^{2/(\beta -2)}<\Vert \nabla \tilde{u}\Vert _{2}^{2}.
\end{equation*}%
In particular, the sign of the corresponding Lagrange multiplier $\lambda$ is positive, the solution $\tilde{u}$ is a ground state, and $\|\nabla \tilde{u}\|_{2}\rightarrow+\infty$ as $|\mu _{\beta }|\rightarrow 0^{+}$.
\end{theorem}

\begin{remark}
\label{r3}$(i)$ Notice that the solution $\tilde{u}$ with negative level is
farther away from the origin than the solution $u^{-}$ with positive level
in Theorem \ref{t2} while the opposite is true in \cite[Theorem 1.2]{CJL}. Meanwhile, it should be noted that as the Van Der Waals coefficient $\mu _{\beta}$ tends to zero, the ground state solution will disappear. Since $\mu _{\beta}<0$, the symmetry of the solution cannot be directly obtained by the method of symmetric rearrangement, which is an open problem. We conjecture that if the solution $u^{-}$ is radially symmetric, then, as $\mu _{\beta}$ approaches zero, the solution $u^{-}$ will converge to the ground state of problem $(HE_{\infty })$, and $E(u^{-})$ will converge to ground state energy $m_{\infty}$.
\newline
$(ii)$ For $N\geq 5$ and $2<\alpha<\beta=4$, the parameter $\mu _{\ast }$ becomes the following more precise form
\begin{equation*}
\mu _{\ast }=\frac{2(\alpha-2)^{3}}{\alpha^{3}\mathcal{S}_{4}m_{\infty }}\left( \frac{4-\alpha }{2}\right) ^{\frac{2}{\alpha-2}}.  \label{e1-8}
\end{equation*}
According to \cite{L1}, we are aware that the ground state energy mapping $c\mapsto m_{\infty }$ is strictly decreasing, one interesting finding is that if the mass is considered as a parameter, then Theorem \ref{t2} holds for large mass $c$.\newline
$(iii)$ If the response function is a Dirac-delta function, i.e. $W_{\alpha
,\beta }(x)=\delta (x)$, then (\ref{e1-2}) is reduced to the following NLS
equation with combined power-type nonlinearities:%
\begin{equation}
\left\{
\begin{array}{ll}
-\Delta u+\lambda u=\mu _{1}|u|^{p-2}u+\mu _{2}|u|^{q-2}u & \quad \text{in}%
\quad \mathbb{R}^{N}, \\
\int_{\mathbb{R}^{N}}|u|^{2}dx=c. &
\end{array}%
\right.  \label{e1-9}
\end{equation}%
When $2<p\leq 2+\frac{4}{N}\leq q<2^{\ast }:=\frac{2N}{N-2}$ with $p\neq q,$ $\mu _{1}\in
\mathbb{R}$ and $\mu _{2}=1$, the existence and stability/instability of
normalized ground states for (\ref{e1-9}) were studied by Soave \cite{S1}.
When $2<p<q=2^{\ast },$ $\mu _{1}>0$ and $\mu _{2}=1,$ the existence,
multiplicity and stability/instability of normalized solutions for (\ref%
{e1-9}) were proved in \cite{JL0,S2,WW}, depending on the parameter $\mu
_{1}, $ which can be viewed as a counterpart of the Br\'{e}zis-Nirenberg
problem in the context of normalized solutions. When $2+\frac{4}{N}%
<p<q\leq2^{\ast },$ $\mu _{1}=1$ and $\mu _{2}<0,$ similar to Theorem \ref{t2},
we can conclude that (\ref{e1-9}) admits two normalized solutions for $|\mu
_{2}|\ $small enough, which is viewed as a complement in \cite{JL0,S1,S2,WW}.
\end{remark}

In Theorem \ref{t2}, in order to find the solution $u^{-}$ with positive
level, \textbf{we propose a novel approach based on the direct minimization of the
energy functional $E$ on the following new constrained manifold}
\begin{equation*}
\mathcal{P}_{D}^{-}\left( c\right) :=\left\{ u\in \mathcal{V}%
_{D}:g_{u}^{\prime }\left( 1\right) =0,g_{u}^{\prime \prime }\left( 1\right)
<0\right\},
\end{equation*}
where
\begin{equation*}
\mathcal{V}_{D}:=\left\{ u\in D(c):\Gamma A(u)^{\frac{\alpha }{2}}<B_{\alpha
}(u)\right\}
\end{equation*}%
with $\Gamma >0$ a constant and
\begin{equation}
D(c):=\left\{ u\in H^{1}(\mathbb{R}^{N}):\int_{\mathbb{R}^{N}}|u|^{2}dx\leq
c\right\} .  \label{e1-3}
\end{equation}
This approach can be viewed as an integration of the approaches in \cite{BM,CJ}.
The advantage of the new constrained manifold is that we can guarantee that an arbitrary $u\in \mathcal{V}_{D}$ can always be
projected onto $\mathcal{P}_{D}^{-}\left( c\right) .$ Furthermore, we do not
work in the radial functions space $H_{r}^{1}(\mathbb{R}^{N})$ and do not
need to consider Palais-Smale sequences, so that we avoid applying the relatively complex
mini-max approach based on a strong topological argument as in \cite%
{BS,CJ,G1,YHSW}. We also avoid excluding the possibility that the critical
point falls on the boundary of the open subset \cite{CJ,YHSW}. To recover
the compactness, a novel strategy is to take advantage of the profile
decomposition of the minimizing sequence and the precise parameter control.
It is worth noting that the mountain pass theorem developed by Jeanjean and
Lu \cite{JL4} can be also used to find the solution with positive level for
Case $(iv)$. However, the energy functional $E$ need to be restricted in $%
H_{r}^{1}(\mathbb{R}^{N}),$ not in $H^{1}(\mathbb{R}^{N})$, and the mass $c$
need to be large.

The second solution $\tilde{u}$ with negative level is
obtained as a global minimizer in Theorem \ref{t2}. The method is mainly based on the
concentration-compactness lemma. We firstly establish a new inequality
estimate to prove that the energy functional $E$ is bounded from below on $%
S(c)$, and then use the fibering map method to prove that the infimum is
negative, which can help us to rule out the vanishing of the minimizing
sequence. Finally, we introduce a new test function of the dilation
transformation, together with the method in \cite{JL5}, to establish the
strict subadditive inequality to exclude the dichotomy of the minimizing
sequence, so that the compactness can be achieved.

Next, we show the relations between the ground state solution and the least action solution.

\begin{definition}
Given $\lambda\in\mathbb{R}$, nontrivial solution $u\in H^{1}(\mathbb{R}^{N})$ to equation (\ref{e1-1}) is called an least action solution if it achieves the infimum of the action functional
\begin{equation*}
E_{\lambda}(u):=E(u)+\frac{\lambda}{2}\int_{\mathbb{R}^{N}}
|u|^{2}dx
\end{equation*}
among all the nontrivial solutions, that is
\begin{equation*}
E_{\lambda}(u)=m_{\lambda}:=\inf\left\{E_{\lambda}(v):v\in H^{1}(\mathbb{R}^{N})\backslash\left\{0\right\},
E_{\lambda}^{\prime}(v)=0\right\}.
\end{equation*}
\end{definition}

\begin{theorem}\label{t4}
Let $\lambda(u)$ be the Lagrange
multiplier corresponding to an arbitrary minimizer $u$ of problem (\ref{e1-2}). Under the assumptions of Theorem \ref{t2}, then the following statements hold.\newline
$(i)$ Any minimizer $u\in S(c)$ of problem (\ref{e1-2}) is a least action solution of (\ref{e1-1}) with $\lambda=\lambda(u)>0$. Moreover,
\begin{equation*}
m_{\lambda}=\sigma(c)+\frac{\lambda c}{2}
\end{equation*}
where
\begin{equation*}
\sigma(c):=\inf_{S(c)}E(u).
\end{equation*}
$(ii)$ For given $\lambda\in \left\{\lambda(u):u\in S(c) \text{is a minimizer of problem (\ref{e1-2})}\right\}$, any least action solution $u_{\lambda}\in H^{1}(\mathbb{R}^{N})$ of (\ref{e1-1}) is a minimizer of problem (\ref{e1-2}), that is
\begin{equation*}
\|u_{\lambda}\|_{2}^{2}=c\,\text{ and } E(u_{\lambda})=\sigma(c).
\end{equation*}
\end{theorem}

Finally, we turn to study the dynamics of standing waves for the Cauchy problem
to (\ref{e1-0}).

\begin{theorem}
\label{t3} Assume that $\mu_{\beta}<0$, $N\geq 3$ and $2<\alpha <\beta <\min \left\{ N,4\right\}$. Then there exists a unique global solution of the Cauchy problem to (\ref{e1-0}) with the initial data $\psi _{0}\in H^{1}(\mathbb{R}^{N})$. What is more, under the assumptions of Theorem \ref{t2}, the set of ground
state
\begin{equation*}
\mathcal{M}_{c}:=\left\{ u\in S(c):E(u)=E(\tilde{u})\right\} \neq \emptyset
\end{equation*}%
is orbitally stable. That is to say, for any $\varepsilon >0$, there exists $%
\delta >0$ such that for any $\varphi \in H^{1}(\mathbb{R}^{N})$ satisfing
\begin{equation*}
dist_{H^{1}}(\varphi ,\mathcal{M}_{c})<\delta ,
\end{equation*}%
the solution $\psi (t,\cdot )$ of the Cauchy problem to (\ref{e1-0}) with
the initial data $\psi (0,\cdot )=\varphi $ satisfies
\begin{equation*}
\sup_{t\in \lbrack 0,T)}dist_{X}(\psi (t,\cdot ),\mathcal{M}%
_{c})<\varepsilon ,
\end{equation*}%
where $T$ is the maximal existence time for $\psi (t,\cdot ).$
\end{theorem}

\begin{remark}
In light of the variational characterization of $u^{-}$ and the geometric structure of the fibering map, we conjecture that standing wave $e^{i\lambda t}u^{-}(x)$ is strongly unstable. As mentioned in the literature \cite{JL4}, the instability will not occur by finite-time blow-up due to the global existence of solutions with any initial data for the Cauchy problem (\ref{e1-0}).
\end{remark}

Moreover, following the argumemts in \cite{TVZ}, via a priori interaction
Morawetz estimate, we give the scattering result for the Cauchy problem to (%
\ref{e1-0}).

\begin{theorem}
\label{t5} Assume that $\mu_{\beta}<0$, $N\geq 3$ and $2<\alpha <\beta <\min \left\{ N,4\right\}$.
If the small mass condition holds, there exists unique global solution $%
\psi _{\pm }\in H^{1}(\mathbb{R}^{N})$ such that
\begin{equation*}
\Vert \psi (t)-e^{i\Delta t}\psi _{\pm }\Vert _{H^{1}(\mathbb{R}%
^{N})}\rightarrow 0\text{ as }t\rightarrow \pm \infty .
\end{equation*}
\end{theorem}

The rest of this paper is organized as follows. In Section 2, we introduce
some preliminary results. In Section 3, we study the multiplicity of
normalized solutions for problem (\ref{e1-2}) and prove Theorems \ref{t2}, \ref{t4} and \ref{t3}.

\section{Preliminary results}

Recall the following well-known Gagliardo-Nirenberg inequality and
Hardy-Littlewood-Sobolev inequality.

\begin{lemma}[Gagliardo-Nirenberg inequality \protect\cite{MS,Y}]
\label{L2-1} Let $N\geq 3$. Then there exists a best constant $\mathcal{S}%
_{\gamma }:=\mathcal{S}_{\gamma }(N,\gamma )>0$ such that
\begin{equation}
\int_{\mathbb{R}^{N}}(|x|^{-\gamma }\ast |u|^{2})|u|^{2}dx\leq \mathcal{S}%
_{\gamma }\Vert \nabla u\Vert _{2}^{\gamma }\Vert u\Vert _{2}^{4-\gamma }.
\label{e2-3}
\end{equation}%
Here,
\begin{equation*}
\mathcal{S}_{\gamma }=\left( \frac{4-\gamma }{\gamma }\right) ^{\gamma /2}%
\frac{4}{(4-\gamma )\Vert Q_{\gamma }\Vert _{2}^{2}},
\end{equation*}%
where $Q_{\gamma }$ is a ground state of the following Hartree equation
\begin{equation*}
-\Delta Q+Q=(|x|^{-\gamma }\ast |Q|^{2})Q\text{ in }\mathbb{R}^{N}.
\end{equation*}
\end{lemma}

\begin{lemma}[Hardy-Littlewood-Sobolev inequality \protect\cite{LL}]
\label{L2-10} Assume that $1<a<b<+\infty $ is such that $\frac{1}{a}+\frac{1%
}{b}+\frac{\gamma }{N}=2.$ Then there exists $\mathcal{C}_{HLS}(N,\gamma)>0$ such that
\begin{equation*}
\int_{\mathbb{R}^{3}}\int_{\mathbb{R}^{3}}\frac{|f(x)||g(y)|}{|x-y|^{\gamma }%
}\leq \mathcal{C}_{HLS}\Vert f\Vert _{a}\Vert g\Vert _{b},\quad \forall f\in
L^{a}(\mathbb{R}^{N})\text{ and }\forall g\in L^{b}(\mathbb{R}^{N}).
\end{equation*}
\end{lemma}

\begin{remark}
For $\gamma=4$, by \cite{GY}, we know that
\begin{equation*}
\mathcal{S}_{4}^{-\frac{1}{2}}=\inf_{u\in D^{1,2}(\mathbb{R}^{N})\backslash\left\{0\right\}}
\frac{\int_{\mathbb{R}^{N}}|\nabla u|^{2}dx}{\left(\int_{\mathbb{R}^{N}}(|x|^{-4}\ast |u|^{2})|u|^{2}dx\right)^{\frac{1}{2}}},
\end{equation*}
where $D^{1,2}(\mathbb{R}^{N}):=\left\{u\in L^{\frac{2N}{N-2}}(\mathbb{R}^{N}):\nabla u\in L^{2}(\mathbb{R}^{N})\right\}$. Moreover, $\mathcal{S}_{4}^{-\frac{1}{2}}$ is achieved if and only if
\begin{equation*}
U_{\varepsilon,y}(x):=
\varepsilon^{\frac{2-N}{2}}U(\frac{x-y}{\varepsilon}),\,
\forall \varepsilon>0,y\in \mathbb{R}^{N},
\end{equation*}
where
\begin{equation*}
U(x):=S^{\frac{4-N}{4}}\mathcal{C}_{HLS}(N,4)^{-\frac{1}{2}}
\frac{[N(N-2)]^{\frac{N-2}{4}}}{(1+|x|^{2})^{\frac{N-2}{2}}}
\end{equation*}
with $\mathcal{C}_{HLS}(N,4)=\pi^{2}\frac{\Gamma(\frac{N-4}{2})}{\Gamma(N-2)}
\left\{\frac{\Gamma(\frac{N}{2})}{\Gamma(N)}\right\}
^{\frac{4-N}{N}}$ and $S$ being the best constant for the Sobolev embedding $D^{1,2}(\mathbb{R}^{N})\hookrightarrow L^{\frac{2N}{N-2}}(\mathbb{R}^{N})$.

\end{remark}

\begin{lemma}
\label{L2-2} Let $u$ be a weak solution to the equation
\begin{equation*}
-\Delta u+\lambda u=\left( |x|^{-\alpha }\ast u^{2}\right) u+\mu _{\beta
}\left( |x|^{-\beta }\ast u^{2}\right) u.
\end{equation*}%
Then we have following Pohozaev identity
\begin{equation}
\frac{N-2}{2}A\left( u\right) +\frac{\lambda N}{2}\int_{\mathbb{R}%
^{N}}|u|^{2}dx=\frac{2N-\alpha }{4}B_{\alpha }\left( u\right) +\frac{\mu
_{\beta }\left( 2N-\beta \right) }{4}B_{\beta }\left( u\right) .
\label{e2-5}
\end{equation}%
Moreover, there holds
\begin{equation*}
A\left( u\right) -\frac{\alpha }{4}B_{\alpha }\left( u\right) -\frac{\mu
_{\beta }\beta }{4}B_{\beta }\left( u\right) =0.
\end{equation*}
\end{lemma}

\begin{proof}
Let $\phi \in C_{0}^{\infty }(\mathbb{R}^{N})$ such that $\phi (x)=1$ on $%
B_{1}(0)$, $\phi (x)=0$ for $x\not\in B_{2}(0)$ and $|\nabla \phi (x)|\leq 2$
for $x\in \mathbb{R}^{N}$. Since the function $\phi $ has compact support,
we can define $\phi _{\tau }:=\phi (\tau x)x\cdot \nabla u(x)\in H^{1}(%
\mathbb{R}^{N})$ for $\tau \in (0,+\infty )$. By testing (\ref{e1-1}) with
the function $\phi _{\tau }$, we have
\begin{equation*}
\int_{\mathbb{R}^{N}}\nabla u\cdot \nabla \phi _{\tau }dx+\lambda \int_{%
\mathbb{R}^{N}}u\phi _{\tau }dx=\int_{\mathbb{R}^{N}}(|x|^{-\alpha }\ast
|u|^{2})u\phi _{\tau }dx+\mu _{\beta }\int_{\mathbb{R}^{N}}(|x|^{-\beta
}\ast |u|^{2})u\phi _{\tau }dx.
\end{equation*}%
By using an integration by parts, we derive
\begin{eqnarray*}
\int_{\mathbb{R}^{N}}\nabla u\cdot \nabla \phi _{\tau }dx &=&\int_{\mathbb{R}%
^{N}}\phi (\tau x)\left( |\nabla u(x)|^{2}+x\cdot \nabla (|\nabla
u|^{2}/2)(x)\right) dx \\
&&+\int_{\mathbb{R}^{N}}(\tau \nabla u(x)\cdot \nabla \phi (\tau x))(x\cdot
\nabla u(x))dx \\
&=&-\frac{1}{2}\int_{\mathbb{R}^{N}}\left( (N-2)\phi (\tau x)+\tau x\cdot
\nabla \phi (\tau x)\right) |\nabla u(x)|^{2}dx \\
&&+\int_{\mathbb{R}^{N}}(\nabla u(x)\cdot \nabla \phi (\tau x))(\tau x\cdot
\nabla u(x))dx.
\end{eqnarray*}%
Since $|(\eta \cdot \zeta )(\eta \cdot \xi )|\leq |\eta |^{2}|\zeta ||\xi |$
for any $\zeta ,\xi ,\eta \in \mathbb{R}^{N}$, we have
\begin{equation*}
|(\nabla u(x)\cdot \nabla \phi (\tau x))(\tau x\cdot \nabla u(x))|\leq
|\nabla u(x)|^{2}|\tau x||\nabla \phi (\tau x)|\leq |\nabla
u(x)|^{2}\sup_{z\in \mathbb{R}^{N}}|z||\nabla \phi (z)|.
\end{equation*}%
Taking a limit as $\tau \rightarrow 0$, by Lebesgue's dominated convergence
theorem, we obtain
\begin{equation*}
\lim_{\tau \rightarrow 0}\int_{\mathbb{R}^{N}}\nabla u\cdot \nabla \phi
_{\tau }dx=-\frac{N-2}{2}\int_{\mathbb{R}^{N}}|\nabla u|^{2}dx.
\end{equation*}%
In light of the definition of $\phi _{\tau }$, we get
\begin{eqnarray*}
\int_{\mathbb{R}^{N}}u\phi _{\tau }dx &=&\int_{\mathbb{R}^{N}}u(x)\phi (\tau
x)x\cdot \nabla u(x)dx \\
&=&\frac{1}{2}\int_{\mathbb{R}^{N}}\phi (\tau x)x\cdot \nabla (|u|^{2})(x)dx
\\
&=&-\frac{1}{2}\int_{\mathbb{R}^{N}}(N\phi (\tau x)+(\tau x)\cdot \nabla
\phi (\tau x))|u(x)|^{2}dx.
\end{eqnarray*}%
By Lebesgue's dominated convergence theorem again, we observe that
\begin{equation*}
\lim_{\tau \rightarrow 0}\int_{\mathbb{R}^{N}}u\phi _{\tau }dx=-\frac{N}{2}%
\int_{\mathbb{R}^{N}}|u|^{2}dx.
\end{equation*}%
Let us now compute the nonlocal term, by symmetry and integration by parts,
we get
\begin{eqnarray*}
&&\int_{\mathbb{R}^{N}}(|x|^{-\alpha }\ast |u|^{2})u\phi _{\tau }dx \\
&=&\int_{\mathbb{R}^{N}}\int_{\mathbb{R}^{N}}|x-y|^{-\alpha }|u(y)|^{2}\phi
(\tau x)x\cdot \nabla (|u|^{2}/2)(x)dxdy \\
&=&\frac{1}{2}\int_{\mathbb{R}^{N}}\int_{\mathbb{R}^{N}}|x-y|^{-\alpha
}\left( |u(y)|^{2}\phi (\tau x)x\cdot \nabla (|u|^{2}/2)(x)+|u(x)|^{2}\phi
(\tau y)y\cdot \nabla (|u|^{2}/2)(y)\right) dxdy \\
&=&-\int_{\mathbb{R}^{N}}\int_{\mathbb{R}^{N}}\frac{|u(x)|^{2}|u(y)|^{2}}{%
2|x-y|^{\alpha }}\left( N\phi (\tau x)+(\tau x)\cdot \nabla \phi (\tau x)-%
\frac{\alpha (x-y)\cdot (x\phi (\tau x)-y\phi (\tau y))}{2|x-y|^{2}}\right)
dxdy.
\end{eqnarray*}%
Then we conclude that
\begin{equation*}
\lim_{\tau \rightarrow 0}\int_{\mathbb{R}^{N}}(|x|^{-\alpha }\ast
|u|^{2})u\phi _{\tau }dx=-\frac{2N-\alpha }{4}\int_{\mathbb{R}%
^{N}}(|x|^{-\alpha }\ast |u|^{2})|u|^{2}dx.
\end{equation*}%
Similarly, we have
\begin{equation*}
\lim_{\tau \rightarrow 0}\int_{\mathbb{R}^{N}}(|x|^{-\beta }\ast
|u|^{2})u\phi _{\tau }dx=-\frac{2N-\beta }{4}\int_{\mathbb{R}%
^{N}}(|x|^{-\beta }\ast |u|^{2})|u|^{2}dx.
\end{equation*}%
Thus (\ref{e2-5}) holds.

On the other hand, multiplying (\ref{e1-1}) by $u$ and integrating on $%
\mathbb{R}^{N}$, we get
\begin{equation}
A(u)+\int_{\mathbb{R}^{N}}|u|^{2}dx=B_{\alpha }\left( u\right) +\mu _{\beta
}B_{\beta }\left( u\right) .  \label{e2-9}
\end{equation}%
Hence, by combining (\ref{e2-5}) and (\ref{e2-9}), we arrive at the
conclusion.
\end{proof}

Next, we give the so-called profile decomposition of bounded sequences,
proposed by G\'{e}rard in \cite{G}, which is crucial to recover the
compactness.

\begin{proposition}
\label{P3-18} Let $\left\{ u_{n}\right\} $ be bounded in $H^{1}(\mathbb{R}%
^{N})$. Then there are sequences $\left\{ \bar{u}_{i}\right\} _{i=0}^{\infty
}\subset H^{1}(\mathbb{R}^{N})$, $\left\{ y_{n}^{i}\right\} _{i=0}^{\infty
}\subset \mathbb{R}^{3}$ for any $n\geq 1$, such that $y_{n}^{0}=0$, $%
|y_{n}^{i}-y_{n}^{j}|\rightarrow \infty $ as $n\rightarrow \infty $ for $%
i\neq j$, and passing to a subsequence, the following conclusions hold for
any $i\geq 0$:
\begin{equation}
u_{n}(\cdot +y_{n}^{i})\rightharpoonup \bar{u}_{i}\quad \text{as }%
n\rightarrow \infty ,  \label{e7-1}
\end{equation}%
with $\lim \sup_{n\rightarrow \infty }\int_{\mathbb{R}^{N}}|z_{n}^{i}|^{q}dx%
\rightarrow 0$ as $i\rightarrow \infty $, where $2<q<2^{\ast}$, $%
z_{n}^{i}:=u_{n}-\sum\limits_{j=0}^{i}\bar{u}_{j}(\cdot -y_{n}^{j})$.
Moreover, there hold
\begin{equation}
\lim_{n\rightarrow \infty }\int_{\mathbb{R}^{N}}|\nabla
u_{n}|^{2}dx=\sum\limits_{j=0}^{i}\int_{\mathbb{R}^{N}}|\nabla \bar{u}%
_{j}|^{2}dx+\lim_{n\rightarrow \infty }\int_{\mathbb{R}^{3}}|\nabla
z_{n}^{i}|^{2}dx,\,  \label{e7-2}
\end{equation}%
and
\begin{equation}
\limsup_{n\rightarrow \infty }\int_{\mathbb{R}^{N}}\int_{\mathbb{R}^{N}}%
\frac{|u_{n}(x)|^{2}|u_{n}(y)|^{2}}{|x-y|^{\gamma }}dxdy=\sum\limits_{j=0}^{%
\infty }\int_{\mathbb{R}^{N}}\int_{\mathbb{R}^{N}}\frac{|\bar{u}_{j}(x)|^{2}|%
\bar{u}_{j}(y)|^{2}}{|x-y|^{\gamma }}dxdy,  \label{e7-5}
\end{equation}%
where $0<\gamma <\min \left\{ N,4\right\} $.
\end{proposition}

\begin{proof}
The proof is based on Vanishing Lemma \cite[Lemma I.1]{L1} and Brezis-Lieb
type results. The (\ref{e7-1})-(\ref{e7-2}) are proved in \cite{HK,MS0,Z}, for the convenience of the reader, we give the details of the proof. We assert that, by passing to a subsequence, there exist $K\in \mathbb{N}\cup \left\{\infty\right\}$ and there is a sequence $\left\{\bar{u}_{i}\right\}_{i=0}^{K}\subset H^{1}(\mathbb{R}^{N})$, for $0\leq i<K+1$ (if $K=\infty$, then $K+1=\infty$) there are sequences $\left\{z_{n}^{i}\right\}\subset H^{1}(\mathbb{R}^{N})$, $\left\{y_{n}^{i}\right\}\subset \mathbb{R}^{N}$ and positive constants $\left\{c_{i}\right\}_{i=0}^{K}$, $\left\{r_{i}\right\}_{i=0}^{K}$ such that $y_{n}^{0}=0$, $r_{0}=0$ and for any $0\leq i<K+1$ one has\newline
$(a)$ $u_{n}(\cdot+y_{n}^{i})\rightharpoonup \bar{u}_{i}$\,in $H^{1}(\mathbb{R}^{N})$ and $u_{n}(\cdot+y_{n}^{i})_{\chi_{B_{n}(0)}}\rightarrow \bar{u}_{i}$ in $L^{2}(\mathbb{R}^{N})$ as $n\rightarrow\infty$;\newline
$(b)$ $\bar{u}_{i}\neq 0$ if $i\geq 1$;\newline
$(c)$ $|y_{n}^{i}-y_{n}^{j}|\geq n-r_{i}-r_{j}$ for $0\leq j\neq i<K+1$ and sufficiently large $n$;\newline
$(d)$ $z_{n}^{-1}:=u_{n}$ and $z_{n}^{i}:=z_{n}^{i-1}-\bar{u}_{i}(\cdot-y_{n}^{i})$ for $n\geq 1$;\newline
$(e)$ $\int_{B_{r_{i}}(y_{n}^{i})}|z_{n}^{i-1}|^{2}dx\geq c_{i}\geq \frac{1}{2}\sup_{y\in \mathbb{R}^{N}}\int_{B_{r_{i}}(y)}|z_{n}^{i-1}|^{2}dx$ for sufficiently large $n$, $r_{i}\geq \max\left\{i,r_{i-1}\right\}$ if $i\geq 1$, and \begin{equation*}
c_{i}=\frac{3}{4}\lim_{r\rightarrow\infty}
\limsup_{n\rightarrow\infty}\sup_{y\in \mathbb{R}^{N}}\int_{B_{r}(y)}|z_{n}^{i-1}|^{2}dx>0.
\end{equation*}

Let $\left\{u_{n}\right\}\subset H^{1}(\mathbb{R}^{N})$ be a bounded sequence, passing to a subsequence we may assume
that $\lim_{n\rightarrow \infty}\|\nabla u_{n}\|_{2}^{2}$ exists and
\begin{equation*}
u_{n}\rightharpoonup \bar{u}_{0}\, \text{ in }\, H^{1}(\mathbb{R}^{N}),
\end{equation*}
\begin{equation*}
u_{n}\chi_{B_{n}(0)}\rightarrow \bar{u}_{0}\,\text{ in }\, L^{2}(\mathbb{R}^{N}),
\end{equation*}
where $\chi_{B_{n}(0)}$ is characteristic function of $B_{n}(0)$. Take $z_{n}^{0}:=u_{n}-\bar{u}_{0}$ and if
\begin{equation*}
\lim_{n\rightarrow\infty}\sup_{y\in \mathbb{R}^{N}}\int_{B_{r}(y)}|z_{n}^{0}|^{2}dx=0
\end{equation*}
for each $r\geq 1$, then we can complete the proof of our claim with $K=0$. Otherwise we have
\begin{equation*}
0<c_{1}:=\frac{3}{4}\lim_{r\rightarrow \infty}\limsup_{n\rightarrow\infty}\sup_{y\in \mathbb{R}^{N}}\int_{B_{r}(y)}|z_{n}^{0}|^{2}dx\leq
\sup_{n\geq 1}\int_{\mathbb{R}^{N}}|z_{n}^{0}|^{2}dx<\infty,
\end{equation*}
and there exists $r_{1}\geq 1$ and, passing to a subsequence, we can find $\left\{y_{n}^{1}\right\}\subset \mathbb{R}^{N}$ such that
\begin{equation}\label{e2-11}
\int_{B_{r_{1}}(y_{n}^{1})}|z_{n}^{0}|^{2}dx\geq c_{1}\geq \frac{1}{2}\sup_{y\in\mathbb{R}^{N}}\int_{B_{r_{1}}
(y)}|z_{n}^{0}|^{2}dx.
\end{equation}
Notice that $\left\{y_{n}^{1}\right\}$ is unbounded and we may suppose that $|y_{n}^{1}|\geq n-r_{1}$. Since $\left\{u_{n}(\cdot+y_{n}^{1})\right\}$ is bounded in $H^{1}(\mathbb{R}^{N})$, up to a subsequence, there exists $\bar{u}_{1}\in H^{1}(\mathbb{R}^{N})$ such that
\begin{equation*}
u_{n}(\cdot+y_{n}^{1})\rightharpoonup \bar{u}_{1}\, \text{ in }\, H^{1}(\mathbb{R}^{N}).
\end{equation*}
It follows from (\ref{e2-11}) that $\bar{u}_{1}\neq 0$, and we can assume that $u_{n}(\cdot+y_{n}^{1})\chi_{B_{n}(0)}\rightarrow \bar{u}_{1}$ in $L^{2}(\mathbb{R}^{N})$. Since
\begin{equation*}
\lim_{n\rightarrow\infty}\left(\int_{\mathbb{R}^{N}}
|\nabla(u_{n}-\bar{u}_{0})(\cdot+y_{n}^{1})|^{2}dx-
\int_{\mathbb{R}^{N}}|\nabla z_{n}^{1}(\cdot+y_{1}^{1})|^{2}dx\right)
=\int_{\mathbb{R}^{N}}|\nabla \bar{u}_{1}|^{2}dx,
\end{equation*}
where $z_{n}^{1}:=z_{n}^{0}-\bar{u}_{1}(\cdot-y_{1}^{1})=
u_{n}-\bar{u}_{0}-\bar{u}_{1}(\cdot-y_{1}^{1})$. Then we get
\begin{equation*}
\lim_{n\rightarrow\infty}\int_{\mathbb{R}^{N}}|\nabla u_{n}|^{2}dx=\int_{\mathbb{R}^{N}}|\nabla \bar{u}_{0}|^{2}dx+\int_{\mathbb{R}^{N}}|\nabla \bar{u}_{1}|^{2}dx+\lim_{n\rightarrow\infty}
\int_{\mathbb{R}^{N}}|\nabla z_{n}^{1}|^{2}dx.
\end{equation*}
If
\begin{equation*}
\lim_{n\rightarrow\infty}\sup_{y\in\mathbb{R}^{N}}
\int_{B_{r}(y)}|z_{n}^{1}|^{2}dx=0
\end{equation*}
for each $r\geq \max\left\{2,r_{1}\right\}$, then we can complete the proof of our claim with $K=1$. Otherwise we have
\begin{equation*}
c_{2}:=\frac{3}{4}\lim_{r\rightarrow \infty}\limsup_{n\rightarrow\infty}\sup_{y\in \mathbb{R}^{N}}\int_{B_{r}(y)}|z_{n}^{1}|^{2}dx>0.
\end{equation*}
Then there is $r_{2}\geq \max\left\{2,r_{1}\right\}$ and, passing to
a subsequence, we find $\left\{y_{n}^{2}\right\}\subset \mathbb{R}^{N}$ such that
\begin{equation}\label{e2-12}
\int_{B_{r_{2}}(y_{n}^{2})}|z_{n}^{1}|^{2}dx\geq c_{2}\geq \frac{1}{2}\sup_{y\in \mathbb{R}^{N}}\int_{B_{r_{2}}(y)}|z_{n}^{1}|^{2}dx
\end{equation}
and $|y_{n}^{2}|\geq n-r_{2}$. Moreover, $|y_{n}^{2}-y_{n}^{1}|\geq n-r_{2}-r_{1}$. Otherwise, $B_{r_{2}}(y_{n}^{2})\subset B_{n}(y_{n}^{1})$ and the convergence $u_{n}(\cdot+y_{n}^{1})\chi_{B_{n}(0)}\rightarrow \bar{u}_{1}$ in $L^{2}(\mathbb{R}^{N})$, which is a contradiction with (\ref{e2-12}). Then, passing to a subsequence, we find $\bar{u}_{2}\neq 0$ such that
\begin{equation*}
z_{n}^{1}(\cdot+y_{n}^{2}), u_{n}(\cdot+y_{n}^{2})\rightharpoonup \bar{u}_{2}\, \text{ in }\, H^{1}(\mathbb{R}^{N}),
\end{equation*}
\begin{equation*}
u_{n}(\cdot+y_{n}^{2})\chi_{B_{n}(0)}\rightarrow \bar{u}_{2}\,\text{ in }\, L^{2}(\mathbb{R}^{N}).
\end{equation*}
Similarly, if
\begin{equation*}
\lim_{n\rightarrow\infty}\sup_{y\in\mathbb{R}^{N}}
\int_{B_{r}(y)}|z_{n}^{2}|^{2}dx=0,
\end{equation*}
for each $r\geq \max\left\{3,r_{2}\right\}$, where $z_{n}^{2}:=z_{n}^{1}-\bar{u}_{2}(\cdot-y_{n}^{2})$, then can finish the proof of our claim with $K=2$. Continuing the above procedure, for each $i\geq 1$, we can find a subsequence of $\left\{u_{n}\right\}$, still denoted by $\left\{u_{n}\right\}$, satisfies $(a)$-$(f)$ and (\ref{e7-2}). Similarly as above, if there exists $i\geq 0$ such that
\begin{equation*}
\lim_{n\rightarrow\infty}\sup_{y\in\mathbb{R}^{N}}
\int_{B_{r}(y)}|z_{n}^{i}|^{2}dx=0
\end{equation*}
for each $r\geq\max\left\{n,r_{i-1}\right\}$, then $K=i$ and we complete proof of the claim. Otherwise, $K=\infty$, by using standard diagonal method and passing to a subsequence, we show that
$(a)$-$(f)$ and (\ref{e7-2}) are satisfied for every $i\geq 0$.

Next, we prove (\ref{e7-5}). Without loss of generality, we assume that $\bar{u}_{i}$ is continuous and compactly supported. For every $i\neq j$, the sequence $\left\{ y_{n}^{i}\right\}
_{i=0}^{\infty }$ satisfies $|y_{n}^{i}-y_{n}^{j}|\rightarrow \infty $ as $%
n\rightarrow \infty $, then we call this sequence $\left\{ y_{n}^{i}\right\}
_{i=0}^{\infty }\subset \mathbb{R}^{3}$ satisfying the orthogonality
condition. Since the sequence $\left\{ u_{n}\right\} $ can be written, up to a
subsequence, as
\begin{equation*}
u_{n}=\sum\limits_{j=0}^{i}\bar{u}_{j}(\cdot -y_{n}^{j})+z_{n}^{i},
\end{equation*}%
and $\limsup_{n\rightarrow \infty }\int_{\mathbb{R}^{3}}|z_{n}^{i}|^{q}dx%
\rightarrow 0$ as $i\rightarrow \infty $, we claim that
\begin{equation}
B_{\gamma }(u_{n})\rightarrow B_{\gamma }\left( \sum\limits_{j=0}^{i}\bar{u}%
_{j}(\cdot -y_{n}^{j})\right) \ \text{as}\ n\rightarrow \infty \ \text{and}\
i\rightarrow \infty .  \label{e7-10}
\end{equation}%
Taking $\bar{u}_{j,n}=\bar{u}_{j}(\cdot -y_{n}^{j})$, to prove (\ref{e7-10})
is equivalent to prove that
\begin{equation}
\int_{\mathbb{R}^{3}}\int_{\mathbb{R}^{3}}\frac{|u_{n}|^{2}|u_{n}|^{2}}{%
|x-y|^{\gamma }}dxdy\rightarrow \int_{\mathbb{R}^{3}}\int_{\mathbb{R}^{3}}%
\frac{\left\vert \sum\limits_{j=0}^{i}\bar{u}_{j,n}\right\vert
^{2}\left\vert \sum\limits_{j=0}^{i}\bar{u}_{j,n}\right\vert ^{2}}{%
|x-y|^{\gamma }}dxdy,  \label{e7-11}
\end{equation}%
as $n\rightarrow \infty $ and $i\rightarrow \infty $. Indeed, to prove (\ref%
{e7-11}), we only need to obtain the following estimates:
\begin{equation}
\lim_{i\rightarrow \infty }\limsup_{n\rightarrow \infty }\int_{\mathbb{R}%
^{3}}\int_{\mathbb{R}^{3}}\frac{|z_{n}^{i}|^{2}|z_{n}^{i}|^{2}}{%
|x-y|^{\gamma }}dxdy=0,  \label{e7-14}
\end{equation}%
\begin{equation}
\lim_{i\rightarrow \infty }\limsup_{n\rightarrow \infty }\int_{\mathbb{R}%
^{3}}\int_{\mathbb{R}^{3}}\frac{|\bar{u}_{j,n}|^{2}|z_{n}^{i}|^{2}}{%
|x-y|^{\gamma }}dxdy=0,  \label{e7-15}
\end{equation}%
\begin{equation}
\lim_{i\rightarrow \infty }\limsup_{n\rightarrow \infty }\int_{\mathbb{R}%
^{3}}\int_{\mathbb{R}^{3}}\frac{|\bar{u}_{j,n}|^{3}|z_{n}^{i}|}{%
|x-y|^{\gamma }}dxdy=0,  \label{e7-16}
\end{equation}%
\begin{equation}
\lim_{i\rightarrow \infty }\limsup_{n\rightarrow \infty }\int_{\mathbb{R}%
^{3}}\int_{\mathbb{R}^{3}}\frac{|\bar{u}_{j,n}||z_{n}^{i}|^{3}}{%
|x-y|^{\gamma }}dxdy=0.  \label{e7-17}
\end{equation}%
It follows from Lemma \ref{L2-10} that
\begin{equation*}
\int_{\mathbb{R}^{3}}\int_{\mathbb{R}^{3}}\frac{%
|z_{n}^{i}|^{2}|z_{n}^{i}|^{2}}{|x-y|^{\gamma }}dxdy\leq \mathcal{C}%
_{HLS}\Vert z_{n}^{i}\Vert _{\frac{4N}{2N-\gamma }}^{2}\Vert z_{n}^{i}\Vert
_{\frac{4N}{2N-\gamma }}^{2}\leq \mathcal{C}_{HLS}\Vert z_{n}^{i}\Vert _{%
\frac{4N}{2N-\gamma }}^{4}.
\end{equation*}%
Since $\limsup_{n\rightarrow \infty }\int_{\mathbb{R}^{3}}|z_{n}^{i}|^{q}dx%
\rightarrow 0$ as $i\rightarrow \infty $, we have $\lim_{i\rightarrow \infty
}\limsup_{n\rightarrow \infty }\Vert z_{n}^{i}\Vert _{\frac{4N}{2N-\gamma }%
}^{4}=0,$ and thus (\ref{e7-14}) is true. By Lemma \ref{L2-10} and Sobolev
embedding theorem, we have
\begin{equation*}
\int_{\mathbb{R}^{3}}\int_{\mathbb{R}^{3}}\frac{|\bar{u}%
_{j,n}|^{2}|z_{n}^{i}|^{2}}{|x-y|^{\gamma }}dxdy\leq \mathcal{C}_{HLS}\Vert
\bar{u}_{j,n}\Vert _{\frac{4N}{2N-\gamma }}^{2}\Vert z_{n}^{i}\Vert _{\frac{%
4N}{2N-\gamma }}^{2}\leq C\Vert \bar{u}_{j,n}\Vert _{H^{1}}^{2}\Vert
z_{n}^{i}\Vert _{\frac{4N}{2N-\gamma }}^{2},
\end{equation*}%
which implies that (\ref{e7-15}) holds, since $\Vert \bar{u}_{j,n}\Vert
_{H^{1}}$ is bounded. For (\ref{e7-16}) and (\ref{e7-17}), we have
\begin{equation*}
\int_{\mathbb{R}^{3}}\int_{\mathbb{R}^{3}}\frac{|\bar{u}%
_{j,n}|^{3}|z_{n}^{i}|}{|x-y|^{\gamma }}dxdy\leq \mathcal{C}_{HLS}\Vert \bar{%
u}_{j,n}\Vert _{\frac{4N}{2N-\gamma }}^{3}\Vert z_{n}^{i}\Vert _{\frac{4N}{%
2N-\gamma }}\leq C\Vert \bar{u}_{j,n}\Vert _{H^{1}}^{3}\Vert z_{n}^{i}\Vert
_{\frac{4N}{2N-\gamma }},
\end{equation*}%
and
\begin{equation*}
\int_{\mathbb{R}^{3}}\int_{\mathbb{R}^{3}}\frac{|\bar{u}%
_{j,n}||z_{n}^{i}|^{3}}{|x-y|^{\gamma }}dxdy\leq \mathcal{C}_{HLS}\Vert \bar{%
u}_{j,n}\Vert _{\frac{4N}{2N-\gamma }}\Vert z_{n}^{i}\Vert _{\frac{4N}{%
2N-\gamma }}^{3}\leq C\Vert \bar{u}_{j,n}\Vert _{H^{1}}\Vert z_{n}^{i}\Vert
_{\frac{4N}{2N-\gamma }}^{3}.
\end{equation*}%
Similar to the arguments of (\ref{e7-14}) and (\ref{e7-15}), we easily
obtain (\ref{e7-16}) and (\ref{e7-17}). So (\ref{e7-11}) holds, and thus (%
\ref{e7-10}) holds.

Finally, to obtain (\ref{e7-5}), it is sufficient to prove that
\begin{equation*}
B_{\gamma}\left(\sum\limits_{j=0}^{\infty}\bar{u}_{j}(\cdot-y_{n}^{j})%
\right)= \sum\limits_{j=0}^{\infty}B_{\gamma}(\bar{u}_{j}(%
\cdot-y_{n}^{j}))+o_{n}(1).
\end{equation*}
By the pairwise orthogonality of the family $\left\{y_{n}^{i}\right%
\}_{i=0}^{\infty}$, and the following the elementary inequality \cite{G}
\begin{equation*}
\left|\left|\sum\limits_{j=0}^{\infty}a_{j}\right|^{2}
-\sum\limits_{j=0}^{\infty}|a_{j}|^{2}\right|\leq C\sum\limits_{j\neq
k}|a_{j}||a_{k}|,
\end{equation*}
we have
\begin{eqnarray*}
&&\left|\int_{\mathbb{R}^{3}}\int_{\mathbb{R}^{3}}\frac{|\sum\limits_{j=0}^{%
\infty}\bar{u}_{j}(x-y_{n}^{j}) |^{2}|\sum\limits_{j=0}^{\infty}\bar{u}%
_{j}(y-y_{n}^{j})|^{2}}{|x-y|^{\gamma}}dxdy-\sum\limits_{j=0}^{\infty}\int_{%
\mathbb{R}^{3}}\int_{\mathbb{R}^{3}} \frac{|\bar{u}_{j}(x-y_{n}^{j})|^{2}|%
\bar{u}_{j}(y-y_{n}^{j})|^{2}}{|x-y|^{\gamma}}dxdy\right|
\end{eqnarray*}
\begin{eqnarray}
&\leq&\sum\limits_{j=0}^{\infty}\sum\limits_{k\neq j}\int_{\mathbb{R}%
^{3}}\int_{\mathbb{R}^{3}} \frac{|\bar{u}_{j}(x-y_{n}^{j})||\bar{u}%
_{k}(x-y_{n}^{k})||\sum\limits_{m=0}^{\infty} \bar{u}_{m}(y-y_{n}^{m})|^{2}}{%
|x-y|^{\gamma}}dxdy  \label{e7-6} \\
&&+\sum\limits_{j=0}^{\infty}\sum\limits_{k\neq j}\int_{\mathbb{R}^{3}}\int_{%
\mathbb{R}^{3}} \frac{|\bar{u}_{j}(y-y_{n}^{j})||\bar{u}_{k}(y-y_{n}^{k})||%
\sum\limits_{m=0}^{\infty} \bar{u}_{m}(x-y_{n}^{m})|^{2}}{|x-y|^{\gamma}}dxdy
\label{e7-7} \\
&&+\sum\limits_{j=0}^{\infty}\sum\limits_{k\neq j}\int_{\mathbb{R}^{3}}\int_{%
\mathbb{R}^{3}} \frac{|\bar{u}_{j}(x-y_{n}^{j})|^{2}|\bar{u}%
_{k}(y-y_{n}^{k})|^{2}}{|x-y|^{\gamma}}dxdy .  \label{e7-8}
\end{eqnarray}
Next, we estimate (\ref{e7-6})--(\ref{e7-8}), respectively. It follows from
Lemma \ref{L2-10} and orthogonality that
\begin{equation*}
(\ref{e7-6})\leq \sum\limits_{j=0}^{\infty}\sum\limits_{k\neq j}\left\||\bar{%
u}_{j}(\cdot-y_{n}^{j})||\bar{u}_{k}(\cdot-y_{n}^{k})|\right\|_{\frac{2N}{%
2N-\gamma}} \left\|\left|\sum\limits_{m=0}^{\infty}\bar{u}%
_{m}(\cdot-y_{n}^{m})\right|^{2}\right\|_{\frac{2N}{2N-\gamma}}\rightarrow
0\ \text{ as }\, n\rightarrow\infty.
\end{equation*}
Similarly, we get $(\ref{e7-7})\rightarrow 0$ as $n\rightarrow\infty$. For (%
\ref{e7-8}), by using the transformation $\tilde{x}:=x-y_{n}^{j}$ and $%
\tilde{y}:=y-y_{n}^{j}$, we have
\begin{eqnarray*}
(\ref{e7-8})&=&\sum\limits_{j=0}^{\infty}\sum\limits_{k\neq j}\int_{\mathbb{R%
}^{3}}\int_{\mathbb{R}^{3}} \frac{|\bar{u}_{j}(\tilde{x})|^{2}|\bar{u}_{k}(%
\tilde{y}-(y_{n}^{k}-y_{n}^{j}))|^{2}}{|\tilde{y}-\tilde{x}|^{\gamma}}d%
\tilde{x}d\tilde{y} \\
&\leq&\sum\limits_{j=0}^{\infty}\sum\limits_{k\neq j}\|\bar{u}_{j}\|_{\frac{%
4N}{2N-\gamma}}^{2}\|\bar{u}_{k}(\cdot-(y_{n}^{k}-y_{n}^{j}))\|_{\frac{4N}{%
2N-\gamma}}^{2}\rightarrow 0\ \text{ as }\ n\rightarrow\infty.
\end{eqnarray*}
Hence, we get
\begin{eqnarray*}
\int_{\mathbb{R}^{3}}\int_{\mathbb{R}^{3}}\frac{|\sum\limits_{j=0}^{\infty}%
\bar{u}_{j}(x-y_{n}^{j}) |^{2}|\sum\limits_{j=0}^{\infty}\bar{u}%
_{j}(y-y_{n}^{j})|^{2}}{|x-y|^{\gamma}}dxdy\rightarrow\sum\limits_{j=0}^{%
\infty}\int_{\mathbb{R}^{3}}\int_{\mathbb{R}^{3}} \frac{|\bar{u}%
_{j}(x-y_{n}^{j})|^{2}|\bar{u}_{j}(y-y_{n}^{j})|^{2}}{|x-y|^{\gamma}}dxdy
\end{eqnarray*}
as $n\rightarrow\infty.$ Therefore, (\ref{e7-5}) holds. We complete the
proof.
\end{proof}

\section{Multiplicity of normalized solutions}

\subsection{The local minimizer with positive level}

Define a subset $\mathcal{V}_{D}\subset D(c)$ by%
\begin{equation*}
\mathcal{V}_{D}:=\left\{ u\in D(c):\Gamma A(u)^{\frac{\alpha }{2}}<B_{\alpha
}(u)\right\} ,
\end{equation*}%
where $D(c)$ is as (\ref{e1-3}) and%
\begin{equation*}
\Gamma =\frac{2(\beta -2)}{\beta -\alpha }\left[ \frac{|\mu _{\beta }|\beta
(\beta -\alpha )\mathcal{S}_{\beta }c^{\frac{4-\beta }{2}}}{4(\alpha -2)}%
\right] ^{\frac{\alpha -2}{\beta -2}}.
\end{equation*}%
Similarly, we define a subset $\mathcal{V}_{S}\subset S(c)$ by%
\begin{equation*}
\mathcal{V}_{S}:=\left\{ u\in S(c):\Gamma A(u)^{\frac{\alpha }{2}}<B_{\alpha
}(u)\right\} .
\end{equation*}%
Clearly, $\mathcal{V}_{S}\subset \mathcal{V}_{D}$. To show that the set $%
\mathcal{V}_{D}$ is not empty, we consider the following problem:%
\begin{equation}\label{e6-5}
\left\{
\begin{array}{ll}
-\Delta u+\lambda u=\left( |x|^{-\alpha }\ast u^{2}\right) u & \text{ in }%
\mathbb{R}^{N}, \\
\int_{\mathbb{R}^{N}}|u|^{2}dx=c>0, &
\end{array}%
\right.   \tag{$HE_{\infty }$}
\end{equation}%
where $N\geq 3$ and $2<\alpha <\min \left\{ N,4\right\} $. From \cite{L1},
we know that (\ref{e6-5}) admits a normalized ground state $\omega _{0}$
with positive energy $m_{\infty }$. Then $m_{\infty }=\inf_{u\in \mathbf{N}%
(c)}E_{\infty }(u)=E_{\infty }(\omega _{0})>0,$ where $E_{\infty }$ is the
energy functional related to (\ref{e6-5}) given by
\begin{equation*}
E_{\infty }(u)=\frac{1}{2}\int_{\mathbb{R}^{N}}|\nabla u|^{2}dx-\frac{1}{4}%
\int_{\mathbb{R}^{N}}\left( |x|^{-\alpha }\ast u^{2}\right) u^{2}dx
\end{equation*}%
and
\begin{equation*}
\mathbf{N}(c)=\left\{ u\in S(c):Q_{\infty }(u)=0\right\}
\end{equation*}%
with $Q_{\infty }(u)=A(u)-\frac{\alpha }{4}B_{\alpha }(u)$. Clearly, it
holds
\begin{equation*}
m_{\infty }=\frac{1}{2}A(\omega _{0})-\frac{1}{4}B_{\alpha }(\omega _{0})=%
\frac{\alpha -2}{2\alpha }A\left( \omega _{0}\right) >0.
\end{equation*}%
For
\begin{equation*}
0<|\mu _{\beta }|<\frac{\alpha (\beta -2)}{\beta (\beta -\alpha )}\left(
\frac{2}{\alpha }\right) ^{\frac{\beta -2}{\alpha -2}}\left( \frac{\alpha
(\beta -2)}{\beta (\alpha -2)}\right) ^{\frac{\beta -2}{2}}\mu _{\ast },
\end{equation*}%
we have $B_{\alpha }(\omega _{0})=\frac{4}{\alpha }A(\omega _{0})=\frac{%
8m_{\infty }}{\alpha -2}>\Gamma A(\omega _{0})^{\frac{\alpha }{2}},$ which
implies that $\omega _{0}\in \mathcal{V}_{S},$ namely, $\mathcal{V}_{D}$ is
also nonempty.

For each $u\in \mathcal{V}_{D}$ and $s>0$, a direct calculation shows that $%
\left\Vert u_{s}\right\Vert _{2}^{2}=\left\Vert u\right\Vert _{2}^{2}\leq c$%
, and
\begin{equation*}
A\left( u_{s}\right) =s^{2}A\left( u\right) ,\quad B_{\alpha }\left(
s_{t}\right) =s^{\alpha }B_{\alpha }\left( u\right) ,\quad B_{\beta }\left(
u_{s}\right) =s^{\beta }B_{\beta }\left( u\right) ,
\end{equation*}%
where $u_{s}$ is as (\ref{e1-4}). Moreover, there hold
\begin{equation*}
g_{u}^{\prime }\left( s\right) =sA\left( u\right) -\frac{\alpha s^{\alpha -1}%
}{4}B_{\alpha }\left( u\right) -\frac{\mu _{\beta }\beta s^{\beta -1}}{4}%
B_{\beta }\left( u\right)
\end{equation*}%
and
\begin{equation*}
g_{u}^{\prime \prime }\left( s\right) =A\left( u\right) -\frac{\alpha
(\alpha -1)s^{\alpha -2}}{4}B_{\alpha }\left( u\right) -\frac{\mu _{\beta
}\beta (\beta -1)s^{\beta -2}}{4}B_{\beta }\left( u\right) ,
\end{equation*}%
where $g_{u}$ is the fibering map defined as (\ref{e1-5}). Notice that
\begin{equation*}
\frac{d}{ds}E\left( u_{s}\right) =\frac{Q\left( u_{s}\right) }{s},
\end{equation*}%
where
\begin{equation*}
Q\left( u\right) :=\frac{d}{ds}|_{s=1}E\left( u_{s}\right) =A\left( u\right)
-\frac{\alpha }{4}B_{\alpha }\left( u\right) -\frac{\mu _{\beta }\beta }{4}%
B_{\beta }\left( u\right) .
\end{equation*}%
Actually the condition $Q\left( u\right) =0$ corresponds to a Pohozaev type
identity, and the set
\begin{equation*}
\mathcal{P}_{D}\left( c\right) :=\left\{ u\in \mathcal{V}_{D}:Q\left(
u\right) =0\right\} =\left\{ u\in \mathcal{V}_{D}:g_{u}^{\prime }\left(
1\right) =0\right\}
\end{equation*}%
appears as a natural constraint. We also recognize that for any $u\in
\mathcal{V}_{D}$, the dilated function $u_{s}\left( x\right) =s^{N/2}u\left(
sx\right) $ belongs to the constraint manifold $\mathcal{P}_{D}\left(
c\right) $ if and only if $s\in \mathbb{R}$ is a critical value of the
fibering map $s\in \left( 0,\infty \right) \mapsto g_{u}\left( s\right) $,
namely, $g_{u}^{\prime }\left( s\right) =0$. Thus, it is natural split $%
\mathcal{P}_{D}\left( c\right) $ into three parts corresponding to local
minima, local maxima and points of inflection. Following \cite{T}, we define
\begin{eqnarray*}
\mathcal{P}_{D}^{+}\left( c\right) &=&\left\{ u\in \mathcal{V}%
_{D}:g_{u}^{\prime }\left( 1\right) =0,g_{u}^{\prime \prime }\left( 1\right)
>0\right\} ; \\
\mathcal{P}_{D}^{-}\left( c\right) &=&\left\{ u\in \mathcal{V}%
_{D}:g_{u}^{\prime }\left( 1\right) =0,g_{u}^{\prime \prime }\left( 1\right)
<0\right\} ; \\
\mathcal{P}_{D}^{0}\left( c\right) &=&\left\{ u\in \mathcal{V}%
_{D}:g_{u}^{\prime }\left( 1\right) =0,g_{u}^{\prime \prime }\left( 1\right)
=0\right\} .
\end{eqnarray*}%
So, $\mathcal{P}_{D}\left( c\right) =\mathcal{P}_{D}^{+}\left( c\right) \cup
\mathcal{P}_{D}^{0}\left( c\right) \cup \mathcal{P}_{D}^{-}\left( c\right) .$
Thus for $u\in \mathcal{P}_{D}\left( c\right) $, we have
\begin{equation}
g_{u}^{\prime \prime }\left( 1\right) =(2-\alpha )A(u)-\frac{\mu _{\beta
}\beta (\beta -\alpha )}{4}B_{\beta }(u)=(2-\beta )A(u)+\frac{\alpha (\beta
-\alpha )}{4}B_{\alpha }(u).\label{e2-7}
\end{equation}%
Similarly, we define $\mathcal{P}_{S}\left( c\right) :=\mathcal{P}%
_{S}^{+}\left( c\right) \cup \mathcal{P}_{S}^{0}\left( c\right) \cup
\mathcal{P}_{S}^{-}\left( c\right) $, where
\begin{eqnarray*}
\mathcal{P}_{S}^{+}\left( c\right) &=&\left\{ u\in \mathcal{V}%
_{S}:g_{u}^{\prime }\left( 1\right) =0,g_{u}^{\prime \prime }\left( 1\right)
>0\right\} ; \\
\mathcal{P}_{S}^{-}\left( c\right) &=&\left\{ u\in \mathcal{V}%
_{S}:g_{u}^{\prime }\left( 1\right) =0,g_{u}^{\prime \prime }\left( 1\right)
<0\right\} ; \\
\mathcal{P}_{S}^{0}\left( c\right) &=&\left\{ u\in \mathcal{V}%
_{S}:g_{u}^{\prime }\left( 1\right) =0,g_{u}^{\prime \prime }\left( 1\right)
=0\right\} ,
\end{eqnarray*}%
and define $\mathcal{P}_{D\backslash S}\left( c\right) :=\mathcal{P}%
_{D\backslash S}^{+}\left( c\right) \cup \mathcal{P}_{D\backslash
S}^{0}\left( c\right) \cup \mathcal{P}_{S}^{-}\left( c\right) $, where%
\begin{eqnarray*}
\mathcal{P}_{D\backslash S}^{+}\left( c\right) &=&\left\{ u\in \mathcal{V}%
_{D\backslash S}:g_{u}^{\prime }\left( 1\right) =0,g_{u}^{\prime \prime
}\left( 1\right) >0\right\} ; \\
\mathcal{P}_{D\backslash S}^{-}\left( c\right) &=&\left\{ u\in \mathcal{V}%
_{D\backslash S}:g_{u}^{\prime }\left( 1\right) =0,g_{u}^{\prime \prime
}\left( 1\right) <0\right\} ; \\
\mathcal{P}_{D\backslash S}^{0}\left( c\right) &=&\left\{ u\in \mathcal{V}%
_{D\backslash S}:g_{u}^{\prime }\left( 1\right) =0,g_{u}^{\prime \prime
}\left( 1\right) =0\right\} .
\end{eqnarray*}

\begin{lemma}
\label{L3-2} Assume that $\mu_{\beta}<0$ and one of the two following conditions hold:\newline
$(i)$ $N\geq 3$ and $2<\alpha <\beta <\min \left\{ N,4\right\};$\newline
$(ii)$ $N\geq 5$ and $2<\alpha<\beta=4$.\newline
Then the energy functional $E$ is bounded from
below by a positive constant and coercive on $\mathcal{P}_{D}^{-}(c).$
\end{lemma}

\begin{proof}
For $u\in \mathcal{P}_{D}^{-}\left( c\right) $, we have
\begin{equation}
A\left( u\right) -\frac{\alpha }{4}B_{\alpha }\left( u\right) -\frac{\mu
_{\beta }\beta }{4}B_{\beta }\left( u\right) =0.  \label{e3-15}
\end{equation}%
Moreover, it follows from (\ref{e2-3}) that
\begin{equation*}
A\left( u\right) \leq A\left( u\right) -\frac{\mu _{\beta }\beta }{4}%
B_{\beta }\left( u\right) =\frac{\alpha }{4}B_{\alpha }\left( u\right) \leq
\frac{\alpha \mathcal{S}_{\alpha }}{4}c^{\frac{4-\alpha }{2}}A(u)^{\frac{%
\alpha }{2}},
\end{equation*}%
which implies that
\begin{equation}
A\left( u\right) \geq \left[ \frac{4}{\alpha \mathcal{S}_{\alpha }c^{\frac{%
4-\alpha }{2}}}\right] ^{\frac{2}{\alpha -2}}.  \label{e2-8}
\end{equation}%
Since
\begin{equation*}
\left( \alpha -2\right) A\left( u\right) >\frac{|\mu _{\beta }|\beta (\beta
-\alpha )}{4}B_{\beta }(u),\label{e3-16}
\end{equation*}%
by (\ref{e3-15})--(\ref{e2-8}), we get
\begin{eqnarray*}
E\left( u\right) &=&\frac{1}{2}A\left( u\right) -\frac{1}{4}B_{\alpha
}\left( u\right) -\frac{\mu _{\beta }}{4}B_{\beta }\left( u\right) \\
&=&\frac{\alpha -2}{2\alpha }A\left( u\right) +\frac{\mu _{\beta }(\beta
-\alpha )}{4\alpha }B_{\beta }\left( u\right) \\
&>&\frac{(\alpha -2)(\beta -2)}{2\alpha \beta }A\left( u\right) \\
&\geq &\frac{(\alpha -2)(\beta -2)}{2\alpha \beta }\left( \frac{4}{\alpha
\mathcal{S}_{\alpha }c^{\frac{4-\alpha }{2}}}\right) ^{\frac{2}{\alpha -2}}
\\
&>&0,
\end{eqnarray*}%
which indicates that $E$ is bounded from below and coercive on $\mathcal{P}%
_{D}^{-}\left( c\right) $. We complete the proof.
\end{proof}
\begin{remark}
\label{R4-1}
In fact, if we define the set $\mathcal{P}^{-}\left( c\right) :=\left\{ \mathcal{P}\left( c\right):g_{u}^{\prime \prime }\left( 1\right)
<0\right\}$, where $\mathcal{P}\left( c\right)$ is the Pohozaev manifold as in (\ref{e1-10}), then the energy functional $E$ is also bounded from below by a positive constant on $\mathcal{P}^{-}(c)$ by using the argument of Lemma \ref{L3-2}.
\end{remark}

Set
\begin{equation*}
s_{u}^{\ast }:=\left[ \frac{4(\beta -2)A(u)}{\alpha (\beta -\alpha
)B_{\alpha }(u)}\right] ^{\frac{1}{\alpha -2}}\text{ and }\widehat{s}_{u}:=%
\left[ \frac{4A(u)}{\alpha B_{\alpha }(u)}\right] ^{\frac{1}{\alpha -2}}.
\end{equation*}

\begin{lemma}
\label{L4-3} Assume that $\mu_{\beta}<0$ and one of the two following conditions hold:\newline
$(i)$ $N\geq 3$ and $2<\alpha <\beta <\min \left\{ N,4\right\};$\newline
$(ii)$ $N\geq 5$ and $2<\alpha<\beta=4$.\newline
Then for each $u\in \mathcal{V}_{D}$ there
exist two positive constants $s_{\ast }^{+}(u)$ and $s_{\ast }^{-}(u)$
satisfying $\widehat{s}_{u}<s_{\ast }^{-}(u)<s_{u}^{\ast }<s_{\ast }^{+}(u)$
such that $u_{s_{\ast }^{\pm }(u)}\in \mathcal{P}_{D}^{\pm }(c)$.
\end{lemma}

\begin{proof}
Define
\begin{equation*}
\widehat{h}(s)=A(u)s^{2-\beta}-\frac{\alpha}{4}B_{\alpha}(u)s^{\alpha-\beta}%
\quad \text{for}\quad s>0.
\end{equation*}%
Clearly, $u_{s}\in \Lambda (c)$ if and only if $\widehat{h}(s)=-\frac{%
\mu_{\beta}\beta}{4}B_{\beta}(u)$. A straightforward calculation shows that $%
\widehat{h}(\widehat{s}_{u})=0,$ $\lim_{s\rightarrow 0^{+}}\widehat{h}%
(s)=\infty $ and $\lim_{s\rightarrow \infty }\widehat{h}(s)=0.$ Note that
\begin{equation*}
\widehat{h}^{\prime }(s)=s^{1-\beta}\left[ (2-\beta)A(u)-\frac{%
\alpha(\alpha-\beta)}{4}B_{\alpha}(u)s^{\alpha-2}\right] .
\end{equation*}%
Then we obtain that $\widehat{h}(s)$ is decreasing when $0<s<s_{u}^{\ast }$
and is increasing when $s>s_{u}^{\ast }$, which implies that
\begin{eqnarray*}
\inf_{s>0}\widehat{h}(s)=\widehat{h}\left(s_{u}^{\ast }\right) =-\frac{%
\alpha-2}{\beta-\alpha}\left[\frac{4(\beta-2)A(u)} {\alpha(\beta-\alpha)B_{%
\alpha}(u)}\right]^{\frac{2-\beta}{\alpha-2}}A(u).
\end{eqnarray*}%
Since $u\in \mathcal{V}_{D}$, we have $\Gamma A(u)^{\frac{\alpha}{2}%
}<B_{\alpha}(u)$. This indicates that
\begin{equation*}
\inf_{s>0}\widehat{h}(s)<-\frac{|\mu_{\beta}|\beta \mathcal{S}_{\beta}}{4}c^{%
\frac{4-\beta}{2}}A(u)^{\frac{\beta}{2}}\leq-\frac{|\mu_{\beta}|\beta}{4}%
B_{\beta}(u),
\end{equation*}%
which implies that there exist two constants $s_{\ast }^{+}(u)$ and $s_{\ast
}^{-}(u)$ satisfying $\widehat{s}_{u}<s_{\ast }^{-}(u)<s_{u}^{\ast }<s_{\ast
}^{+}(u)$ such that
\begin{equation*}
\widehat{h}(s_{\ast }^{\pm }(u))=-\frac{|\mu_{\beta}|\beta}{4}B_{\beta}(u).
\end{equation*}%
Namely, $u_{s_{\ast }^{\pm }(u)}\in \mathcal{P}(c).$ By a calculation on the
second order derivatives, we find that
\begin{equation*}
g_{u_{s_{\ast }^{-}(u)}}^{\prime \prime }(1)=(s_{\ast }^{-}(u))^{\beta +1}
\widehat{h}^{\prime }(s_{\ast }^{-}(u))<0\text{ and }g_{u_{s_{\ast
}^{+}\left( u\right) }}^{\prime \prime }\left( 1\right) =\left( s_{\ast
}^{+}\left( u\right) \right) ^{\beta+1}\widehat{h}^{\prime }\left( s_{\ast
}^{+}\left( u\right) \right) >0.
\end{equation*}%
These imply that $u_{s_{\ast }^{\pm }\left( u\right) }\in \mathcal{P}%
_{D}^{\pm }\left( c\right) $. This completes the proof.
\end{proof}

\begin{corollary}
\label{C3-5} If $u\in \mathcal{V}_{D}$ and $Q(u)\leq 0$, then there exists $%
0<s_{\ast}^{-}(u)\leq1$ such that $Q(u_{s_{\ast}^{-}(u)})=0$.
\end{corollary}

\begin{proof}
Suppose that the contrary, there exists $s_{\ast}^{-}(u)>1$ such that $%
Q(u_{s_{\ast}^{-}(u)})=0$. Then we have
\begin{equation*}
\widehat{h}(s_{\ast }^{-}(u))=-\frac{|\mu_{\beta}|\beta}{4}B_{\beta}(u).
\end{equation*}%
Since $\widehat{h}(s)$ is decreasing when $0<s<s_{u}^{\ast }$, then
\begin{equation*}
A(u)-\frac{\alpha}{4}B_{\alpha}(u)=\widehat{h}(1)>-\frac{|\mu_{\beta}|\beta}{%
4}B_{\beta}(u),
\end{equation*}
and so $Q(u)>0$. This is a contradiction.
\end{proof}

\begin{lemma}
\label{L4-6} Assume that $0<|\mu _{\beta }|<\min \left\{ \mu _{\ast },\frac{%
\alpha (\beta -2)}{\beta (\beta -\alpha )}\left( \frac{2}{\alpha }\right) ^{%
\frac{\beta -2}{\alpha -2}}\left( \frac{\alpha (\beta -2)}{\beta (\alpha -2)}%
\right) ^{\frac{\beta -2}{2}}\mu _{\ast }\right\} $. Then $m(c):=\inf_{u\in
\mathcal{P}_{D}^{-}(c)}E(u)>0$ is achieved at $u^{-}\in D(c)$.
\end{lemma}

\begin{proof}
Clearly, $m(c)>0$ by Lemma \ref{L3-2}. Let $\left\{ u_{n}\right\} \subset
\mathcal{P}_{D}^{-}(c)$ be a bounded and non-vanishing minimizing sequence
to $m(c)$. Recall that $\omega _{0}$ is a ground state normalized solution
for (\ref{e6-5}). Since
\begin{equation*}
0<|\mu _{\beta }|<\frac{\alpha (\beta -2)}{\beta (\beta -\alpha )}\left(
\frac{2}{\alpha }\right) ^{\frac{\beta -2}{\alpha -2}}\left( \frac{\alpha
(\beta -2)}{\beta (\alpha -2)}\right) ^{\frac{\beta -2}{2}}\mu _{\ast },
\end{equation*}%
we have $\Gamma A(\omega _{0})^{\frac{\alpha }{2}}<B_{\alpha }(\omega _{0})$%
, and by Lemma \ref{L4-3}, there exists
\begin{equation*}
1<s_{\ast }^{-}(\omega _{0})<\left( \frac{\beta -2}{\beta -\alpha }\right) ^{%
\frac{1}{\alpha -2}}
\end{equation*}%
such that $\Vert (\omega _{0})_{s_{\ast }^{-}(\omega _{0})}\Vert _{2}^{2}=c$
and $(\omega _{0})_{s_{\ast }^{-}(\omega _{0})}\in \mathcal{P}_{D}^{-}(c)$.
Moreover, a direct calculation shows that
\begin{eqnarray*}
E\left( \left( \omega _{0}\right) _{s_{\ast }^{-}(\omega _{0})}\right) &=&%
\frac{\beta -2}{2\beta }A\left( \left( \omega _{0}\right) _{s_{\ast
}^{-}(\omega _{0})}\right) -\frac{\beta -\alpha }{4\beta }B_{\alpha }\left(
\left( \omega _{0}\right) _{s_{\ast }^{-}(\omega _{0})}\right) \\
&=&\frac{\beta -2}{2\beta }\left( s_{\ast }^{-}(\omega _{0})\right)
^{2}A\left( \omega _{0}\right) -\frac{\beta -\alpha }{4\beta }\left( s_{\ast
}^{-}(\omega _{0})\right) ^{\beta }B_{\alpha }\left( \omega _{0}\right) \\
&=&\left[ \frac{\beta -2}{2\beta }\left( s_{\ast }^{-}(\omega _{0})\right)
^{2}-\frac{\beta -\alpha }{\alpha \beta }\left( s_{\ast }^{-}(\omega
_{0})\right) ^{\beta }\right] A\left( \omega _{0}\right) \\
&<&\frac{\alpha (\beta -2)^{2}m_{\infty }}{\beta ^{2}(\alpha -2)}\left[
\frac{\alpha (\beta -2)}{\beta (\beta -\alpha )}\right] ^{2/(\beta -2)},
\end{eqnarray*}%
which implies that
\begin{equation}
A\left( u_{n}\right) +o_{n}(1)<\frac{2\alpha \beta }{(\alpha -2)(\beta -2)}%
m(c)<\frac{2\alpha ^{2}(\beta -2)m_{\infty }}{\beta (\alpha -2)^{2}}\left[
\frac{\alpha (\beta -2)}{\beta (\beta -\alpha )}\right] ^{2/(\beta -2)}.\label{e5-7}
\end{equation}

By applying Proposition \ref{P3-18}, we can find a profile decomposition of $%
\left\{ u_{n}\right\} $ satisfying (\ref{e7-1})--(\ref{e7-5}). Define the
index set
\begin{equation*}
I:=\left\{ i\geq 0:u_{i}^{-}\neq 0\right\} .
\end{equation*}%
By using Lemma \ref{L3-2} and (\ref{e7-5}), we get that $I\neq \emptyset $.
Now, we show that $Q(u_{i}^{-})\leq 0$ for some $i\geq 0$. Suppose that the
contrary, if $Q(u_{i}^{-})>0$ for all $i\in I$, then it follows from
Proposition \ref{P3-18} that
\begin{eqnarray*}
\lim \sup_{n\rightarrow \infty }\left[ \frac{\alpha }{4}B_{\alpha }(u_{n})+%
\frac{\mu _{\beta }\beta }{4}B_{\beta }(u_{n})\right]  &=&\lim
\sup_{n\rightarrow \infty }A(u_{n}) \\
&\geq &\sum\limits_{i=0}^{\infty }A(u_{i}^{-})=\sum\limits_{i\in
I}A(u_{i}^{-}) \\
&>&\sum\limits_{i=0}^{\infty }\left[ \frac{\alpha }{4}B_{\alpha }(u_{i}^{-})+%
\frac{\mu _{\beta }\beta }{4}B_{\beta }(u_{i}^{-})\right]  \\
&=&\lim \sup_{n\rightarrow \infty }\left[ \frac{\alpha }{4}B_{\alpha
}(u_{n})+\frac{\mu _{\beta }\beta }{4}(B_{\beta }(u_{n})-B_{\beta}(u_{n}-u_{i}^{-}))\right]\\
&\geq&\lim \sup_{n\rightarrow \infty }\left[ \frac{\alpha }{4}B_{\alpha
}(u_{n})+\frac{\mu _{\beta }\beta }{4}B_{\beta }(u_{n})\right] ,
\end{eqnarray*}%
which is a contradiction. Thus, there holds
\begin{equation*}
I^{-}:=\left\{ i\in I:Q(u_{i}^{-})\leq 0\right\} \neq \emptyset .
\end{equation*}%
We claim that $\Gamma A(u_{i}^{-})^{\frac{\alpha }{2}}<B_{\alpha }(u_{i}^{-})
$ for $i\in I^{-}$. Otherwise, if $\Gamma A(u_{i}^{-})^{\frac{\alpha }{2}%
}\geq B_{\alpha }(u_{i}^{-})$ and $Q(u_{i}^{-})\leq 0$, we have
\begin{equation*}
A(u_{i}^{-})\leq A(u_{i}^{-})-\frac{\mu _{\beta }\beta }{4}B_{\beta
}(u_{i}^{-})\leq \frac{\alpha }{4}B_{\alpha }(u_{i}^{-})\leq \frac{\alpha }{4%
}\Gamma A(u_{i}^{-})^{\frac{\alpha }{2}},
\end{equation*}%
which implies that
\begin{equation*}
A(u_{i}^{-})\geq \left( \frac{4}{\alpha \Gamma }\right) ^{\frac{2}{\alpha -2}%
}.
\end{equation*}%
By using (\ref{e5-7}), we get
\begin{equation*}
\left( \frac{\beta -\alpha }{\beta -2}\right) ^{\frac{2}{\alpha -2}}\left[
\frac{4(\alpha -2)}{|\mu _{\beta }|\beta (\beta -\alpha )\mathcal{S}_{\beta
}c^{\frac{4-\beta }{2}}}\right] ^{\frac{2}{\beta -2}}\leq A(u_{i}^{-})\leq
\lim_{n\rightarrow \infty }A(u_{n})\leq \frac{2\alpha ^{2}(\beta
-2)m_{\infty }}{\beta (\alpha -2)^{2}}\left[ \frac{\alpha (\beta -2)}{\beta
(\beta -\alpha )}\right] ^{\frac{2}{\beta -2}}.
\end{equation*}%
This is a contradiction, since $0<|\mu _{\beta }|<\mu _{\ast }$.

For simplicity, let us denote $u^{-}:=u_{i}^{-}$ for $i\in I^{-}$. By
Corollary \ref{C3-5}, there exists $0<s_{\ast }^{-}(u^{-})\leq 1$ such that $%
\Vert u_{s_{\ast }^{-}(u^{-})}^{-}\Vert _{2}^{2}\leq c$ and $u_{s_{\ast
}^{-}(u^{-})}^{-}\in \mathcal{P}_{D}^{-}(c)$. Then we have
\begin{eqnarray}
0 &<&m(c)  \notag \\
&\leq &E(u_{s_{\ast }^{-}(u^{-})}^{-})  \notag \\
&=&\frac{\alpha -2}{2\alpha }A\left( u_{s_{\ast }^{-}(u^{-})}^{-}\right) -%
\frac{|\mu _{\beta }|(\beta -\alpha )}{4\alpha }B_{\beta }\left( u_{s_{\ast
}^{-}(u^{-})}^{-}\right)  \notag \\
&=&(s_{\ast }^{-}(u^{-}))^{2}\frac{\alpha -2}{2\alpha }A(u^{-})-(s_{\ast
}^{-}(u^{-}))^{\beta }\frac{|\mu _{\beta }|(\beta -\alpha )}{4\alpha }%
B_{\beta }(u^{-})  \label{e4-5} \\
&\leq &\frac{\alpha -2}{2\alpha }A(u^{-})-\frac{|\mu _{\beta }|(\beta
-\alpha )}{4\alpha }B_{\beta }(u^{-})  \label{e4-6} \\
&=&\lim_{n\rightarrow \infty }\left[ \frac{\alpha -2}{2\alpha }\left(
A(u_{n})-A(u_{n}-u^{-})\right) -\frac{|\mu _{\beta }|(\beta -\alpha )}{%
4\alpha }\left( B_{\beta }(u_{n})-B_{\beta }(u_{n}-u^{-})\right) \right]
\notag \\
&=&\lim_{n\rightarrow \infty }\left[ \frac{\alpha -2}{2\alpha }A(u_{n})-%
\frac{|\mu _{\beta }|(\beta -\alpha )}{4\alpha }B_{\beta }(u_{n})\right]
\notag \\
&&+\lim_{n\rightarrow \infty }\left[ \frac{|\mu _{\beta }|(\beta -\alpha )}{%
4\alpha }B_{\beta }(u_{n}-u^{-})-\frac{\alpha -2}{2\alpha }A(u_{n}-u^{-})%
\right]  \label{e4-7} \\
&\leq &\lim_{n\rightarrow \infty }\left[ \frac{\alpha -2}{2\alpha }A(u_{n})-%
\frac{|\mu _{\beta }|(\beta -\alpha )}{4\alpha }B_{\beta }(u_{n})\right]
\label{e4-8} \\
&=&\lim_{n\rightarrow \infty }E(u_{n})=m(c).  \notag
\end{eqnarray}%
So $s_{\ast }^{-}(u^{-})=1$ and $E(u^{-})=m(c)$. This implies that $%
u_{n}(\cdot +y_{n}^{i})\rightarrow u^{-}$ in $H^{1}(\mathbb{R}^{N})$ and
therefore $I$ is a singleton set.

We now provide the details of the derivation of (\ref{e4-5})-(\ref{e4-8}).%
\newline
$\bullet $ $(\ref{e4-5})\Rightarrow (\ref{e4-6}):$ Let
\begin{equation*}
f(s)=\frac{\alpha -2}{2\alpha }A(u^{-})s^{2}-\frac{|\mu _{\beta }|(\beta
-\alpha )}{4\alpha }B_{\beta }(u^{-})s^{\beta }.
\end{equation*}%
A straightforward calculation shows that $f(s)$ is increasing when $%
0<s<s_{\max }$ and is decreasing when $s>s_{\max }$, where
\begin{equation*}
s_{\max }:=\left[ \frac{4(\alpha -2)A(u^{-})}{|\mu _{\beta }|\beta (\beta
-\alpha )B_{\beta }(u^{-})}\right] ^{\frac{1}{\beta -2}}.
\end{equation*}%
Next, we give the estimate of $s_{\max }$. It follows from (\ref{e5-7}) that
\begin{eqnarray*}
s_{\max } &=&\left[ \frac{4(\alpha -2)A(u^{-})}{|\mu _{\beta }|\beta (\beta
-\alpha )B_{\beta }(u^{-})}\right] ^{\frac{1}{\beta -2}} \\
&\geq &\left[ \frac{4(\alpha -2)A(u^{-})}{|\mu _{\beta }|\beta (\beta
-\alpha )\mathcal{S}_{\beta }A(u^{-})^{\frac{\beta }{2}}c^{\frac{4-\beta }{2}%
}}\right] ^{\frac{1}{\beta -2}} \\
&\geq &\lim_{n\rightarrow \infty }\left[ \frac{4(\alpha -2)}{|\mu _{\beta
}|\beta (\beta -\alpha )\mathcal{S}_{\beta }A(u_{n})^{\frac{\beta -2}{2}}c^{%
\frac{4-\beta }{2}}}\right] ^{\frac{1}{\beta -2}}
\end{eqnarray*}%
\begin{eqnarray*}
&=&\lim_{n\rightarrow \infty }\left[ \frac{4(\alpha -2)}{|\mu _{\beta
}|\beta (\beta -\alpha )\mathcal{S}_{\beta }c^{\frac{4-\beta }{2}}}\right] ^{%
\frac{1}{\beta -2}}A(u_{n})^{-\frac{1}{2}} \\
&>&\left[ \frac{4(\alpha -2)}{|\mu _{\beta }|\alpha (\beta -2)\mathcal{S}%
_{\beta }c^{\frac{4-\beta }{2}}}\right] ^{\frac{1}{\beta -2}}\left( \frac{%
\beta (\alpha -2)^{2}}{2\alpha ^{2}(\beta -2)m_{\infty }}\right) ^{1/2}.
\end{eqnarray*}%
Thus for $0<|\mu _{\beta }|<\mu _{\ast }$, we get
\begin{equation*}
s_{\max }>\left[ \frac{4(\alpha -2)}{|\mu _{\beta }|\alpha (\beta -2)%
\mathcal{S}_{\beta }c^{\frac{4-\beta }{2}}}\right] ^{\frac{1}{\beta -2}%
}\left( \frac{\beta (\alpha -2)^{2}}{2\alpha ^{2}(\beta -2)m_{\infty }}%
\right) ^{1/2}\geq 1.
\end{equation*}%
Since $0<s_{\ast }^{-}(u^{-})\leq 1$, then we get $f(s_{\ast
}^{-}(u^{-}))\leq f(1)$.\newline
$\bullet $ $(\ref{e4-7})\Rightarrow (\ref{e4-8}):$ We only need to prove
that
\begin{equation*}
\lim_{n\rightarrow \infty }\left[ \frac{|\mu _{\beta }|(\beta -\alpha )}{%
4\alpha }B_{\beta }(u_{n}-u^{-})-\frac{\alpha -2}{2\alpha }A(u_{n}-u^{-})%
\right] \leq 0.
\end{equation*}%
Otherwise, for $n$ sufficiently large, we have
\begin{eqnarray*}
\frac{\alpha -2}{2\alpha }A(u_{n}-u^{-}) &<&\frac{|\mu _{\beta }|(\beta
-\alpha )}{4\alpha }B_{\beta }(u_{n}-u^{-}) \\
&\leq &\frac{|\mu _{\beta }|(\beta -\alpha )}{4\alpha }\mathcal{S}_{\beta
}A(u_{n}-u^{-})^{\frac{\beta -2}{2}}\Vert u_{n}-u^{-}\Vert _{2}^{4-\beta } \\
&\leq &\frac{|\mu _{\beta }|(\beta -\alpha )}{4\alpha }\mathcal{S}_{\beta
}A(u_{n}-u^{-})^{\frac{\beta -2}{2}}c^{\frac{4-\beta }{2}},
\end{eqnarray*}%
and so
\begin{eqnarray*}
|\mu _{\beta }| &>&\lim_{n\rightarrow \infty }\frac{2(\alpha -2)}{(\beta
-\alpha )\mathcal{S}_{\beta }A(u_{n}-u^{-})^{\frac{\beta -2}{2}}c^{\frac{%
4-\beta }{2}}} \\
&\geq &\lim_{n\rightarrow \infty }\frac{2(\alpha -2)}{(\beta -\alpha )%
\mathcal{S}_{\beta }A(u_{n})^{\frac{\beta -2}{2}}c^{\frac{4-\beta }{2}}} \\
&>&\frac{2\beta (\alpha -2)}{\alpha (\beta -2)\mathcal{S}_{\beta }c^{\frac{%
4-\beta }{2}}}\left[ \frac{\beta (\alpha -2)^{2}}{2\alpha ^{2}(\beta
-2)m_{\infty }}\right] ^{\frac{\beta -2}{2}} \\
&=&\frac{2\beta }{4}\left( \frac{\beta -2}{\beta -\alpha }\right) ^{\frac{%
\beta -2}{\alpha -2}}\mu _{\ast },
\end{eqnarray*}%
which is a contradiction, since $0<|\mu _{\beta }|<\mu _{\ast }$. We
complete the proof.
\end{proof}

\begin{lemma}
\label{L4-4} If $0<|\mu _{\beta }|<\min \left\{ \mu _{\ast },\frac{\alpha
(\beta -2)}{\beta (\beta -\alpha )}\left( \frac{2}{\alpha }\right) ^{\frac{%
\beta -2}{\alpha -2}}\left( \frac{\alpha (\beta -2)}{\beta (\alpha -2)}%
\right) ^{\frac{\beta -2}{2}}\mu _{\ast }\right\} $, then there holds
\begin{equation*}
m(c)=\inf_{u\in \mathcal{P}_{D}^{-}}E(u)=\inf_{u\in \mathcal{P}%
_{S}^{-}}E(u)<\inf_{u\in \mathcal{P}_{D\backslash S}^{-}}E(u).
\end{equation*}
\end{lemma}

\begin{proof}
Suppose by contradiction that there is $\bar{u}\in \mathcal{P}_{D\backslash
S}^{-}$ and
\begin{equation*}
m(c)=E(\bar{u})\leq \inf_{u\in \mathcal{P}_{S}^{-}}E(u).
\end{equation*}%
Since $\bar{u}$ is a local minimizer for $E$ on $\mathcal{P}_{D}^{-}$, there
is a Lagrange multiplier $\nu \in \mathbb{R}$ such that
\begin{equation*}
\left\langle E^{\prime }(\bar{u}),\phi \right\rangle +\nu \left( \int_{%
\mathbb{R}^{N}}\nabla \bar{u}\nabla \phi dx+\frac{|\mu _{\beta }|\beta }{2}%
\int_{\mathbb{R}^{N}}(|x|^{-\beta }\ast |\bar{u}|^{2})\bar{u}\phi dx-\frac{%
\alpha }{2}\int_{\mathbb{R}^{N}}(|x|^{-\alpha }\ast |\bar{u}|^{2})\bar{u}%
\phi dx\right) =0
\end{equation*}%
for any $\phi \in H^{1}(\mathbb{R}^{N})$. Namely, $\bar{u}$ is a weak
solution to
\begin{equation*}
-(1+\nu )\Delta \bar{u}+|\mu _{\beta }|(1+\frac{\nu \beta }{2})(|x|^{-\beta
}\ast |\bar{u}|^{2})\bar{u}-(1+\frac{\nu \alpha }{2})(|x|^{-\alpha }\ast |%
\bar{u}|^{2})\bar{u}=0.
\end{equation*}%
Moreover, $\bar{u}$ satisfies the following Nehari-type identity
\begin{equation}
(1+\nu )A(\bar{u})+|\mu _{\beta }|(1+\frac{\nu \beta }{2})B_{\beta }(\bar{u}%
)-(1+\frac{\nu \alpha }{2})B_{\alpha }(\bar{u})=0,  \label{e5-1}
\end{equation}%
and Pohozaev type identity
\begin{equation}
\frac{(N-2)(1+\nu )}{2}A(\bar{u})+\frac{|\mu _{\beta }|\left( 2N-\beta
\right) (1+\frac{\nu \beta }{2})}{4}B_{\beta }(\bar{u})-\frac{\left(
2N-\alpha \right) (1+\frac{\nu \alpha }{2})}{4}B_{\alpha }(\bar{u})=0.
\label{e5-2}
\end{equation}

We will prove through separating it into two cases.

\textbf{Case $(I)$:} $N\geq 3$ and $2<\alpha <\beta <\min \left\{ N,4\right\}$.

$(I-a):\nu \leq 0$. Combining (\ref{e5-1}) with (\ref{e5-2}), we get
\begin{equation}
(1+\frac{\nu \beta }{2})(4-\beta )|\mu _{\beta }|B_{\beta }(\bar{u})=(1+%
\frac{\nu \alpha }{2})(4-\alpha )B_{\alpha }(\bar{u}).  \label{e5-3}
\end{equation}%
Since $(4-\beta )|\mu _{\beta }|B_{\beta }(\bar{u})>0,(4-\alpha )B_{\alpha }(%
\bar{u})>0$ and $1+\frac{\nu \beta }{2}\leq 1+\frac{\nu \alpha }{2}$
(equality occurs if and only if $\nu =0$), we have $1+\frac{\nu \beta }{2}%
\neq 0$ and $1+\frac{\nu \alpha }{2}\neq 0$. By using (\ref{e5-3}), one has
\begin{equation}
|\mu _{\beta }|B_{\beta }(\bar{u})=\frac{(1+\frac{\nu \alpha }{2})(4-\alpha )%
}{(1+\frac{\nu \beta }{2})(4-\beta )}B_{\alpha }(\bar{u}).  \label{e5-6}
\end{equation}%
Since $\bar{u}\in \mathcal{P}_{D}^{-}$, we have
\begin{equation*}
\frac{\beta (\beta -2)}{4}|\mu _{\beta }|B_{\beta }(\bar{u})<\frac{\alpha
(\alpha -2)}{4}B_{\alpha }(\bar{u}),
\end{equation*}%
and then
\begin{equation*}
\frac{\beta (\beta -2)(1+\frac{\nu \alpha }{2})(4-\alpha )}{4(1+\frac{\nu
\beta }{2})(4-\beta )}<\frac{\alpha (\alpha -2)}{4},
\end{equation*}%
which is a contradiction, since
\begin{equation*}
\beta (\beta -2)(4-\alpha )(1+\frac{\nu \alpha }{2})>\alpha (\alpha
-2)(4-\beta )(1+\frac{\nu \beta }{2}).
\end{equation*}

$(I-b):\nu >0$. Since $\bar{u}\in \mathcal{P}_{D}^{-}$, we have
\begin{equation}
A(\bar{u})+\frac{|\mu _{\beta }|\beta }{4}B_{\beta }(\bar{u})-\frac{\alpha }{%
4}B_{\alpha }(\bar{u})=0.  \label{e5-4}
\end{equation}%
By (\ref{e5-1}) and (\ref{e5-4}), we get
\begin{equation*}
(4-\beta +\nu \beta )|\mu _{\beta }|B_{\beta }(\bar{u})=(4-\alpha +\nu
\alpha )B_{\alpha }(\bar{u}).
\end{equation*}%
Since $(4-\beta )+\nu \beta >0$, we have
\begin{equation}
|\mu _{\beta }|B_{\beta }(\bar{u})=\frac{4-\alpha +\nu \alpha }{4-\beta +\nu
\beta }B_{\alpha }(\bar{u}).  \label{e5-5}
\end{equation}%
Combining (\ref{e5-6}) and (\ref{e5-5}), we get
\begin{equation}
\frac{4-\alpha +\nu \alpha }{4-\beta +\nu \beta }=\frac{(1+\frac{\nu \alpha
}{2})(4-\alpha )}{(1+\frac{\nu \beta }{2})(4-\beta )}.  \label{e5-8}
\end{equation}%
But a direct calculation shows that
\begin{eqnarray*}
&&(1+\frac{\nu \alpha }{2})(4-\alpha )(4-\beta +\nu \beta )-(1+\frac{\nu
\beta }{2})(4-\beta )(4-\alpha +\nu \alpha ) \\
&=&\frac{\nu (\beta -\alpha )}{2}\left[ 8-(4-\alpha )(4-\beta )\right] +%
\frac{\nu ^{2}\alpha \beta (\beta -\alpha )}{2}>0,
\end{eqnarray*}%
since $(4-\alpha )(4-\beta )<4$. This contradicts with (\ref{e5-8}).

\textbf{Case $(II)$:} $N\geq 5$ and $2<\alpha <\beta=4$. Combining (\ref{e5-1}) with (\ref{e5-2}), we get
\begin{equation}
(1+\frac{\nu \alpha }{2})(4-\alpha )B_{4}(\bar{u})=0,  \label{e5-3}
\end{equation}%
and so $\nu=-\frac{2}{\alpha}$. Since $\bar{u}\in \mathcal{P}_{D}^{-}$, we have
\begin{equation}\label{e3-20}
(\alpha-2)A(\bar{u})>|\mu _{\beta
}|(4-\alpha )B_{4}(\bar{u}).
\end{equation}
By (\ref{e5-1}), we get
\begin{equation*}
(1+\nu )A(\bar{u})+|\mu _{\beta }|(1+2\nu )B_{4}(\bar{u}%
)=0.
\end{equation*}
Substitute $\nu=-\frac{2}{\alpha}$ into the above formula, one has
\begin{equation*}
(\alpha-2)A(\bar{u})=|\mu _{\beta }|(4-\alpha)B_{4}(\bar{u}),
\end{equation*}
which contradicts with (\ref{e3-20}). We complete the proof.
\end{proof}

\begin{lemma}
\label{L4-2} Assume that $\mu_{\beta}<0$ and one of the two following conditions hold:\newline
$(i)$ $N\geq 3$ and $2<\alpha <\beta <\min \left\{ N,4\right\};$\newline
$(ii)$ $N\geq 5$ and $2<\alpha<\beta=4$.\newline
Then $\mathcal{P}_{S}^{0}(c)=\emptyset$.
\end{lemma}

\begin{proof}
Suppose that $\mathcal{P}_{S}^{0}(c)\neq \emptyset $. Then for $u\in
\mathcal{P}_{S}^{0}(c)$, it follows from (\ref{e2-7}) that%
\begin{equation*}
(\alpha -2)A(u)=\frac{|\mu _{\beta }|\beta (\beta -\alpha )}{4}B_{\beta
}(u)\leq \frac{|\mu _{\beta }|\beta (\beta -\alpha )\mathcal{S}_{\beta }}{4}%
c^{\frac{4-\beta }{2}}A(u)^{\frac{\beta }{2}},
\end{equation*}%
which indicates that
\begin{equation}
A(u)\geq \left[ \frac{4(\alpha -2)}{|\mu _{\beta }|\beta (\beta -\alpha )%
\mathcal{S}_{\beta }c^{\frac{4-\beta }{2}}}\right] ^{\frac{2}{\beta -2}}.
\label{e4-1}
\end{equation}%
On the other hand, by using (\ref{e2-7}) and the fact of $\Gamma A(u)^{\frac{%
\alpha }{2}}<B_{\alpha }(u)$, we have
\begin{equation*}
A(u)<\left[ \frac{4(\alpha -2)}{|\mu _{\beta }|\beta (\beta -\alpha )%
\mathcal{S}_{\beta }c^{\frac{4-\beta }{2}}}\right] ^{\frac{2}{\beta -2}},
\end{equation*}%
which is a contradiction with (\ref{e4-1}). We complete the proof.
\end{proof}

Furthermore, following the arguments of Soave \cite{S1}, we have the
following lemma.

\begin{lemma}
\label{L4-14} If $\mathcal{P}_{S}^{0}(c)=\emptyset $, then $\mathcal{P}_{S}(c)$
is a natural constraint manifold, namely, if $u^{-}$ is a critical point
for $E|_{\mathcal{P}_{S}(c)}$, then $u^{-}$ is a critical point for $%
E|_{S(c)}$.
\end{lemma}

\begin{proof}
If $u^{-}\in \mathcal{P}_{S}(c)$ is a critical point for $E|_{\mathcal{P}%
_{S}(c)}$, then by the Lagrange multipliers rule, there exists $\lambda^{-}%
,\nu^{-}$ such that
\begin{equation*}
\left\langle E^{\prime }(u^{-}),\varphi \right\rangle +\lambda^{-}\int_{%
\mathbb{R}^{N}}u^{-}\varphi dx+\nu^{-}\left\langle Q^{\prime }(u^{-}%
),\varphi \right\rangle =0,\,\forall \varphi \in H^{1}(\mathbb{R}^{N}).
\end{equation*}%
Then $\bar{u}$ solves
\begin{equation*}
-(1+\nu^{-})\Delta u^{-}+\lambda^{-}u^{-}+|\mu _{\beta }|(1+\frac{%
\nu^{-}\beta }{2})(|x|^{-\beta }\ast |u^{-}|^{2})u^{-}-(1+\frac{\nu^{-}\alpha }{2})(|x|^{-\alpha }\ast |u^{-}|^{2})u^{-}=0,
\end{equation*}%
together with the Pohozaev identity, there holds
\begin{equation*}
(1+\nu^{-})A(u^{-})+\frac{|\mu _{\beta }|\beta (1+\frac{\nu^{-}\beta }{%
2})}{4}B_{\beta }(u^{-})-\frac{\alpha (1+\frac{\nu^{-}\alpha }{2})}{4}%
B_{\alpha }(u^{-})=0.
\end{equation*}%
Since $Q(u^{-})=0$ and $u^{-}\not\in \mathcal{P}_{S}^{0}(c)$, it follows
from
\begin{equation*}
\frac{\nu^{-}}{2}\left[ A\left( u^{-}\right) +\frac{|\mu _{\beta }|\beta
(\beta -1)}{4}B_{\beta }\left( u^{-}\right) -\frac{\alpha (\alpha -1)}{4}%
B_{\alpha }\left( u^{-}\right) \right] =0
\end{equation*}%
that $\nu^{-}=0$. We complete the proof.
\end{proof}

\begin{lemma}
\label{L4-5} The sign of the corresponding Lagrange multiplier to the solution of problem (\ref{e1-2}) is positive.
\end{lemma}

\begin{proof}
Let $u$ is a solution for problem (\ref{e1-2}), by the Lagrange
multipliers rule, there exists $\lambda\in \mathbb{R}$ such that for
every $\varphi \in H^{1}(\mathbb{R}^{N})$,
\begin{equation*}
\int_{\mathbb{R}^{N}}\nabla u\nabla \varphi dx+\lambda\int_{%
\mathbb{R}^{N}}u\varphi dx-\mu _{\beta }\int_{\mathbb{R}%
^{N}}(|x|^{-\beta }\ast |u|^{2})|u|\varphi -\int_{\mathbb{R}%
^{N}}(|x|^{-\alpha }\ast |u|^{2})|u|\varphi =0.
\end{equation*}%
In particular, we have
\begin{equation*}
\lambda c=B_{\alpha }(u)-A(u)+\mu _{\beta }B_{\beta }(u).  \label{e3-17}
\end{equation*}%
Since $Q(u)=0$, we have
\begin{equation*}
A\left( u\right) -\frac{\mu _{\beta }\beta }{4}B_{\beta }\left( u
\right) -\frac{\alpha }{4}B_{\alpha }\left( u\right) =0.  \label{e3-18}
\end{equation*}%
Then, we have
\begin{equation*}
\lambda=\frac{1}{c}\left[ \frac{4-\alpha }{\alpha }A\left( u%
\right) +\frac{|\mu _{\beta }|(\beta -\alpha)}{\alpha }B_{\beta }\left( u\right) %
\right] >0.
\end{equation*}%
We complete the proof.
\end{proof}

\begin{theorem}
\label{T4-1} Assume that $\mu_{\beta}<0$ and one of the two following conditions hold:\newline
$(i)$ $N\geq 3$ and $2<\alpha <\beta <\min \left\{ N,4\right\};$\newline
$(ii)$ $N\geq 5$ and $2<\alpha<\beta=4$.\newline
Then for
\begin{equation*}
0<|\mu _{\beta }|<\min \left\{ \mu _{\ast },\frac{\alpha (\beta -2)}{\beta
(\beta -\alpha )}\left(
\frac{2}{\alpha }\right) ^{\frac{\beta -2}{\alpha -2}}\left( \frac{\alpha (\beta -2)}{\beta (\alpha -2)}\right) ^{%
\frac{\beta -2}{2}}\mu _{\ast }\right\} ,
\end{equation*}%
the energy functional $E|_{S(c)}$ has a critical point $u^{-}\in S(c)$
with positive level $m(c).$
\end{theorem}

\begin{proof}
By Lemmas \ref{L4-6}--\ref{L4-14}, we easily obtain the result.
\end{proof}

\subsection{The global minimizer with negative level}

\begin{lemma}
\label{L5-0} Assume that $\mu_{\beta}<0$ and one of the two following conditions hold:\newline
$(i)$ $N\geq 3$ and $2<\alpha <\beta <\min \left\{ N,4\right\};$\newline
$(ii)$ $N\geq 5$ and $2<\alpha<\beta=4$.\newline
Then the energy functional $E$ is bounded from below and
coercive on $S(c)$ for all $c>0.$
\end{lemma}

\begin{proof}
By the H\"{o}lder's inequality, we have
\begin{eqnarray*}
&&\int_{\mathbb{R}^{N}}\int_{\mathbb{R}^{N}}\frac{|u(x)|^{2}|u(y)|^{2}}{%
|x-y|^{\alpha }}dxdy \\
&=&\int_{\mathbb{R}^{N}}\int_{\mathbb{R}^{N}}\frac{\left(
|u(x)|^{2}|u(y)|^{2}\right) ^{\theta }}{|x-y|^{\theta \beta }}\frac{\left(
|u(x)|^{2}|u(y)|^{2}\right) ^{1-\theta }}{|x-y|^{1-\theta }}dxdy \\
&\leq &\left( \int_{\mathbb{R}^{N}}\int_{\mathbb{R}^{N}}\frac{%
|u(x)|^{2}|u(y)|^{2}}{|x-y|^{\beta }}dxdy\right) ^{\theta }\left( \int_{%
\mathbb{R}^{N}}\int_{\mathbb{R}^{N}}\frac{|u(x)|^{2}|u(y)|^{2}}{|x-y|}%
dxdy\right) ^{1-\theta },
\end{eqnarray*}%
where $0<\theta =\frac{\alpha -1}{\beta -1}<1$. Next, we make use of Young's
inequality,%
\begin{equation*}
ab\leq \varepsilon a^{r}+\varepsilon ^{-\frac{r^{\prime }}{r}}b^{r^{\prime }}%
\text{ for any }a,b,\varepsilon >0,
\end{equation*}%
where $r>1$ and $r^{\prime }=\frac{r}{r-1}$. Then
\begin{eqnarray*}
&&\left( \int_{\mathbb{R}^{N}}\int_{\mathbb{R}^{N}}\frac{|u(x)|^{2}|u(y)|^{2}%
}{|x-y|^{\beta }}dxdy\right) ^{\theta }\left( \int_{\mathbb{R}^{N}}\int_{%
\mathbb{R}^{N}}\frac{|u(x)|^{2}|u(y)|^{2}}{|x-y|}dxdy\right) ^{1-\theta } \\
&\leq &\varepsilon \int_{\mathbb{R}^{N}}\int_{\mathbb{R}^{N}}\frac{%
|u(x)|^{2}|u(y)|^{2}}{|x-y|^{\beta }}dxdy+\varepsilon ^{-\frac{\theta }{%
1-\theta }}\int_{\mathbb{R}^{N}}\int_{\mathbb{R}^{N}}\frac{%
|u(x)|^{2}|u(y)|^{2}}{|x-y|}dxdy.
\end{eqnarray*}
This shows that
\begin{equation*}
B_{\alpha }(u)\leq \varepsilon B_{\beta }(u)+\varepsilon ^{-\frac{\theta }{%
1-\theta }}B_{1}(u).
\end{equation*}%
Choosing $\varepsilon =\frac{|\mu _{\beta }|}{2}$, we have
\begin{eqnarray*}
E(u) &=&\frac{1}{2}A(u)+\frac{|\mu _{\beta }|}{4}B_{\beta }(u)-\frac{1}{4}%
B_{\alpha }(u) \\
&\geq &\frac{1}{2}A(u)+\frac{|\mu _{\beta }|}{4}B_{\beta }(u)-\frac{%
\varepsilon }{4}B_{\beta }(u)-\frac{\varepsilon ^{-\frac{\theta }{1-\theta }}%
}{4}B_{1}(u) \\
&\geq &\frac{1}{2}A(u)+\frac{|\mu _{\beta }|}{8}B_{\beta }(u)-\frac{1}{4}%
\left( \frac{|\mu _{\beta }|}{2}\right) ^{-\frac{\theta }{1-\theta }}B_{1}(u)
\\
&\geq &\frac{1}{2}A(u)-\frac{1}{4}\left( \frac{|\mu _{\beta }|}{2}\right) ^{-%
\frac{\theta }{1-\theta }}\Vert u\Vert _{2}^{3}A(u)^{1/2},
\end{eqnarray*}%
which implies that the energy functional $E$ is bounded from below and
coercive on $S(c)$ for all $c>0$. This completes the proof.
\end{proof}

\begin{lemma}
\label{L5-1} Assume that $\mu_{\beta}<0$ and one of the two following conditions hold:\newline
$(i)$ $N\geq 3$ and $2<\alpha <\beta <\min \left\{ N,4\right\};$\newline
$(ii)$ $N\geq 5$ and $2<\alpha<\beta=4$.\newline
Then for each $u\in \mathcal{V}_{S}$, there holds $%
\inf_{s>0}E(u_{s})<0.$
\end{lemma}

\begin{proof}
Let $u\in \mathcal{V}_{S}.$ Then for any $s>0$, we have
\begin{equation*}
E(u_{s})=s^{\beta }\left[ \frac{s^{2-\beta }}{2}A\left( u\right) -\frac{%
s^{\alpha -\beta }}{4}B_{\alpha }\left( u\right) +\frac{|\mu _{\beta }|}{4}%
B_{\beta }\left( u\right) \right] .
\end{equation*}%
Let
\begin{equation*}
k(s):=\frac{s^{2-\beta }}{2}A\left( u\right) -\frac{s^{\alpha -\beta }}{4}%
B_{\alpha }\left( u\right) .
\end{equation*}%
Clearly, $E(u_{s})=0$ if and only if $k\left( s\right) +\frac{|\mu _{\beta }|%
}{4}B_{\beta }\left( u\right) =0.$ It is easily seen that $%
k(s_{0})=0,\lim_{s\rightarrow 0^{+}}k(s)=\infty $ and $\lim_{s\rightarrow
\infty }k(s)=0,$ where $s_{0}:=\left( \frac{2A(u)}{B_{\alpha }(u)}\right)
^{1/(\alpha -2)}$. By calculating the derivative of $k(s)$, we get
\begin{equation*}
k^{\prime }(s)=s^{1-\beta }\left[ \frac{(\beta -\alpha )s^{\alpha -2}}{4}%
B_{\alpha }(u)-\frac{\beta -2}{2}A(u)\right] .
\end{equation*}%
This indicates that $k(s)$ is decreasing when $0<s<\left( \frac{2(\beta
-2)A(u)}{(\beta -\alpha )B_{\alpha }(u)}\right) ^{1/(\alpha -2)}$ and is
increasing when $s>\left( \frac{2(\beta -2)A(u)}{(\beta -\alpha )B_{\alpha
}(u)}\right) ^{1/(\alpha -2)}$, and so we have
\begin{eqnarray*}
\inf_{s>0}k(s) &=&-\frac{\alpha -2}{2(\beta -\alpha )}\left[ \frac{(\beta
-\alpha )B_{\alpha }(u)}{2(\beta -2)A(u)}\right] ^{\frac{\beta -2}{\alpha -2}%
}A(u) \\
&<&-\frac{|\mu _{\beta }|\mathcal{S}_{\beta }}{4}c^{\frac{4-\beta }{2}}A(u)^{%
\frac{\beta }{2}} \\
&<&-\frac{|\mu _{\beta }|}{4}B_{\beta }\left( u\right) ,
\end{eqnarray*}%
which implies that there exist two constants $\hat{s}^{(i)}>0$ $(i=1,2)$
satisfying
\begin{equation*}
\hat{s}^{(1)}<\left( \frac{2(\beta -2)A(u)}{(\beta -\alpha )B_{\alpha }(u)}%
\right) ^{1/(\alpha -2)}<\hat{s}^{(2)}
\end{equation*}%
such that $E\left( u_{\hat{s}^{(i)}}\right) =0$ for $i=1,2.$ Moreover, we
get
\begin{equation*}
E\left[ \left( \frac{2(\beta -2)A(u)}{(\beta -\alpha )B_{\alpha }(u)}\right)
^{1/(\alpha -2)}\right] <0,
\end{equation*}%
which implies that $\inf_{s>0}E(u_{s})<0.$ The proof is complete.
\end{proof}

By Lemmas \ref{L5-0} and \ref{L5-1}, we have%
\begin{equation*}
-\infty <\sigma (c):=\inf_{u\in S(c)}E(u)<0.
\end{equation*}

\begin{lemma}
\label{L5-2} $(i)$ For any $c>c^{\prime }>0$ one has
\begin{equation*}
\sigma (c)\leq \frac{c}{c^{\prime }}\sigma (c^{\prime }).
\end{equation*}%
If $\sigma (c^{\prime })$ is achieved, then the inequality is strict.\newline
$(ii)$ The function $c\mapsto \sigma (c)$ is non-increasing and continuous.
\end{lemma}

\begin{proof}
$(i)$ Let $z:=\frac{c}{c^{\prime }}>1$. For any $\varepsilon >0$, there
exists $u\in S(c^{\prime })$ such that $E(u)\leq \sigma (c^{\prime
})+\varepsilon $. Clearly, $v(x):=u(z^{-1/N}x)\in S(c)$ and then
\begin{eqnarray}
\sigma (c) &\leq &E(v)=z^{\frac{2N-\alpha }{N}}E(u)+\frac{1}{2}z^{\frac{N-2}{%
N}}\left( 1-z^{\frac{N+2-\alpha }{N}}\right) A(u)+\frac{|\mu _{\beta }|}{4}%
z^{\frac{2N-\beta }{N}}\left( 1-z^{\frac{\beta -\alpha }{N}}\right) B_{\beta
}(u)  \notag \\
&<&zE(u)\leq \frac{c}{c^{\prime }}(\sigma (c^{\prime })+\varepsilon ).
\label{e8-2}
\end{eqnarray}%
Since $\varepsilon >0$ is arbitrary, the inequality $\sigma (c)\leq \frac{c}{%
c^{\prime }}\sigma (c^{\prime })$ holds. If $\sigma (c^{\prime })$ is
achieved, for example, at some $u\in S(c^{\prime })$, then we can let $%
\varepsilon =0$ in (\ref{e8-2}) and thus the strict inequality follows.

$(ii)$ It is clear that $\sigma (c)$ is non-increasing. To show the
continuity, we assume that $c_{n}\rightarrow c$ as $n\rightarrow \infty $.
It follows from the definition of $\sigma \left( c\right) $ that for any $%
\varepsilon >0$, there exists $u_{n}\in S(c_{n})$ such that
\begin{equation}
E(u_{n})\leq \sigma \left( c\right) +\varepsilon .  \label{e8-20}
\end{equation}%
Let $v_{n}:=\frac{c}{c_{n}}u_{n}$. Taking into account that $v_{n}\in S(c)$
and $\frac{c}{c_{n}}\rightarrow 1$, we have
\begin{equation}
\sigma \left( c\right) \leq E(v_{n})=E(u_{n})+o(1).  \label{e8-21}
\end{equation}%
Combining (\ref{e8-20}) and (\ref{e8-21}), one has
\begin{equation*}
\sigma \left( c\right) \leq \sigma \left( c_{n}\right) +\varepsilon +o(1).
\end{equation*}%
By inverting the argument, we obtain
\begin{equation*}
\sigma \left( c_{n}\right) \leq \sigma \left( c\right) +\varepsilon +o(1).
\end{equation*}%
Therefore, since $\varepsilon >0$ is arbitrary, we deduce that $\sigma
\left( c_{n}\right) \rightarrow \sigma \left( c\right) $ as $n\rightarrow
\infty $. The proof is complete.
\end{proof}

\begin{theorem}
\label{T5-1} Assume that $\mu_{\beta}<0$ and one of the two following conditions hold:\newline
$(i)$ $N\geq 3$ and $2<\alpha <\beta <\min \left\{ N,4\right\};$\newline
$(ii)$ $N\geq 5$ and $2<\alpha<\beta=4$.\newline
Then for
\begin{equation*}
0<|\mu _{\beta }|<\min \left\{ \mu _{\ast },\frac{\alpha (\beta -2)}{\beta
(\beta -\alpha )}\left( \frac{2}{\alpha }\right) ^{\frac{\beta -2}{\alpha -2}%
}\left( \frac{\alpha (\beta -2)}{\beta (\alpha -2)}\right) ^{\frac{\beta -2}{%
2}}\mu _{\ast }\right\} ,
\end{equation*}%
then $\sigma (c)=\inf_{u\in S(c)}E(u)$ is achieved at $\tilde{u}\in S(c).$
\end{theorem}

\begin{proof}
Fix
\begin{equation*}
0<|\mu _{\beta }|<\min \left\{ \mu _{\ast },\frac{\alpha (\beta -2)}{\beta
(\beta -\alpha )}\left( \frac{2}{\alpha }\right) ^{\frac{\beta -2}{\alpha -2}%
}\left( \frac{\alpha (\beta -2)}{\beta (\alpha -2)}\right) ^{\frac{\beta -2}{%
2}}\mu _{\ast }\right\} ,
\end{equation*}%
and let $\left\{ u_{n}\right\} \subset S(c)$ be any minimizing sequence with
respect to $\sigma (c)$. It is clear that $\left\{ u_{n}\right\} $ is
bounded in $H^{1}(\mathbb{R}^{N})$. Since $\sigma (c)<0$, we deduce that $%
\left\{ u_{n}\right\} $ is non-vanishing, namely
\begin{equation}
\lim_{n\rightarrow \infty }\left( \sup_{y\in \mathbb{R}^{N}}%
\int_{B_{1}(y)}|u_{n}|^{2}dx\right) >0.  \label{e8-1}
\end{equation}%
Indeed, if (\ref{e8-1}) were not true, then $u_{n}\rightarrow 0$ in $L^{r}(%
\mathbb{R}^{N})$ for $2<q<2^{\ast }$ by \cite[Lemma I.1]{L}, which implies
that%
\begin{equation*}
B_{\alpha }(u_{n})\rightarrow 0
\text{ as }n\rightarrow \infty .
\end{equation*}%
Thus we have
\begin{equation*}
0>\sigma (c)=\lim_{n\rightarrow \infty }E(u_{n})\geq\frac{1}{2}\int_{\mathbb{R}%
^{N}}|\nabla u_{n}|^{2}dx\geq 0,
\end{equation*}%
which is a contradiction. Since $\left\{ u_{n}\right\} $ is non-vanishing,
there exists a sequence $\{y_{n}\}\subset \mathbb{R}^{N}$ and a nontrivial
function $\tilde{u}\in H^{1}(\mathbb{R}^{N})$ such that up to a subsequence $%
u_{n}(\cdot +y_{n})\rightharpoonup \tilde{u}$ and $u_{n}(\cdot
+y_{n})\rightarrow \tilde{u}$ a.e. on $\mathbb{R}^{N}$. Set $c^{\prime
}:=\Vert \tilde{u}\Vert _{2}^{2}\leq c$ and $w_{n}:=u_{n}(\cdot +y_{n})-%
\tilde{u}$. Then we have $\lim_{n\rightarrow \infty }\Vert w_{n}\Vert
_{2}^{2}=c-c^{\prime }$ and
\begin{equation}
\sigma (c)=\lim_{n\rightarrow \infty }E(u_{n})=\lim_{n\rightarrow \infty
}E(w_{n}+\tilde{u})=E(\tilde{u})+\lim_{n\rightarrow \infty }E(w_{n}),
\label{e8-3}
\end{equation}%
where we have used the Brezis-Lieb Lemma. Let $\zeta _{n}:=\Vert w_{n}\Vert
_{2}^{2}$ for each $n\in \mathbb{N}^{+}$. If $\lim_{n\rightarrow \infty
}\zeta _{n}>0$, then we have $0<c^{\prime }<c$. According to the definition
of $\sigma (\zeta _{n})$, we obtain
\begin{equation}
\lim_{n\rightarrow \infty }E(w_{n})\geq \lim_{n\rightarrow \infty }\sigma
(\zeta _{n})=\sigma (c-c^{\prime }).  \label{e8-4}
\end{equation}%
It follows from (\ref{e8-3}), (\ref{e8-4}) and Lemma \ref{L5-2} that
\begin{equation*}
\sigma (c)\geq E(\tilde{u})+\sigma (c-c^{\prime })\geq \sigma (c^{\prime
})+\sigma (c-c^{\prime })\geq \frac{c^{\prime }}{c}\sigma (c)+\frac{%
c-c^{\prime }}{c}\sigma (c)=\sigma (c),
\end{equation*}%
which implies that $E(\tilde{u})=\sigma (c)-\sigma (c-c^{\prime }),$ leading
to $E(\tilde{u})=\sigma (c^{\prime }).$ This shows that $\sigma (c^{\prime
}) $ is achieved at $\tilde{u}\in S(c^{\prime })$. Moreover, there holds
\begin{equation*}
\sigma (c)\geq \sigma (c^{\prime })+\sigma (c-c^{\prime })>\frac{c^{\prime }%
}{c}\sigma (c)+\frac{c-c^{\prime }}{c}\sigma (c)=\sigma (c),
\end{equation*}%
which is a contradiction. So, $c=c^{\prime }$. Then $\tilde{u}\in S(c)$ and
\begin{equation*}
\sigma (c)=\lim_{n\rightarrow \infty }E(u_{n})\geq \lim_{n\rightarrow \infty
}E(u_{n})-\lim_{n\rightarrow \infty }E(w_{n})=E(\tilde{u})\geq \sigma (c),
\end{equation*}%
where we have used $\lim_{n\rightarrow \infty }E(w_{n})\geq 0,$ since $%
\lim_{n\rightarrow \infty }B_{\alpha }(w_{n})=0$. This implies that $\sigma (c)$ is achieved at $\tilde{%
u}$. This completes the proof.
\end{proof}

\textbf{We are ready to prove Theorem \ref{t2}:} The results can be obtained
by Theorems \ref{T4-1} and \ref{T5-1}. Moreover, by Remark \ref{R4-1}, we note that $\tilde{u}\in \mathcal{P}^{+}\left( c\right) :=\left\{ \mathcal{P}\left( c\right):g_{u}^{\prime \prime }\left( 1\right)
>0\right\}$, where $\mathcal{P}\left( c\right)$ is the Pohozaev manifold as in (\ref{e1-10}). Then by (\ref{e2-7}), we have
\begin{equation}\label{e3-22}
\Vert \nabla \tilde{u}\Vert _{2}^{2}>\left( \frac{4(\alpha -2)}{|\mu _{\beta
}|\mathcal{S}_{\beta }\beta (\beta -\alpha )c^{\frac{4-\beta }{2}}}\right) ^{%
\frac{2}{\beta -2}}.
\end{equation}%
And since
\begin{equation*}
\Vert \nabla u^{-}\Vert _{2}^{2}<\frac{2\alpha ^{2}(\beta -2)m_{\infty }}{%
\beta (\alpha -2)^{2}}\left[ \frac{\alpha (\beta -2)}{\beta (\beta -\alpha )}%
\right] ^{2/(\beta -2)},
\end{equation*}%
for $|\mu _{\beta }|<\mu _{\ast }$, we have
\begin{equation*}
\Vert \nabla u^{-}\Vert _{2}^{2}<\Vert \nabla \tilde{u}\Vert _{2}^{2}.
\end{equation*}
Again by (\ref{e3-22}), we get
\begin{equation*}
\Vert \nabla \tilde{u}\Vert _{2}^{2}>\left( \frac{4(\alpha -2)}{|\mu _{\beta
}|\mathcal{S}_{\beta }\beta (\beta -\alpha )c^{\frac{4-\beta }{2}}}\right) ^{%
\frac{2}{\beta -2}}\rightarrow +\infty,\,\text{ as }|\mu _{\beta }|\rightarrow 0^{+}.
\end{equation*}

Next, we give the mountain pass type characterization of the nontrivial solutions of (\ref{e1-1}) with $\lambda>0$.

\begin{proposition}\label{p3-1}
For any nontrivial critical point $u_{\lambda}\in H^{1}(\mathbb{R}^{N})$ of action functional $E_{\lambda}$, any $\delta>0$ and any $M>0$, then there exist a number $T=T(u_{\lambda},\delta,M)>0$ and a continuous path $\kappa: [0,T]\rightarrow H^{1}(\mathbb{R}^{N})$ satisfying\newline
$(i)$ $\kappa(0)=0, E_{\lambda}(\kappa(T))<-1, \max_{t\in [0,T]}E_{\lambda}(\kappa(t))=E_{\lambda}(u_{\lambda})$;\newline
$(ii)$ $\kappa(\tau)=u_{\lambda}$ for some $\tau\in (0,T)$, and $E_{\lambda}(\kappa(t))<E_{\lambda}(u_{\lambda})$ for any $t\in[0,T]$ such that $\|\kappa(t)-u_{\lambda}\|_{H^{1}(\mathbb{R}^{N})}\geq\delta$;\newline
$(iii)$ $\varpi(t):=\|\kappa(t)\|_{2}^{2}$ is a strictly increasing continuous function with $\varpi(T)>M$.
\end{proposition}

\begin{proof}
The proof is similar to \cite[Lemma 1.8]{JL5}, we omit it here.
\end{proof}

\textbf{The proof of Theorem \ref{t4}:} $(i)$ Let $u_{\lambda}$ be nontrivial critical point of $E_{\lambda}$. For a fixed $\delta>0$ and $M:=c>0$, let $\kappa: [0,T]\rightarrow H^{1}(\mathbb{R}^{N})$ be the continuous path defined in Proposition \ref{p3-1}. By Proposition \ref{p3-1}, there is $t_{0}\in(0,T)$ such that $\|\kappa(t_{0})\|_{2}^{2}=c$ and further
\begin{eqnarray*}
 E_{\lambda}(u_{\lambda})=\max_{t\in [0,T]}E_{\lambda}(\kappa(t))&\geq& E_{\lambda}(\kappa(t_{0}))\\
 &=&E(\kappa(t_{0}))+\frac{\lambda}{2}
 \int_{\mathbb{R}^{N}}|\kappa(t_{0})|^{2}dx\\
 &\geq&\sigma(c)+\frac{\lambda c}{2}.
\end{eqnarray*}
$(ii)$ By $(i)$, we know that any least action solution $u_{\lambda}\in H^{1}(\mathbb{R}^{N})$ of (\ref{e1-1}) satisfies
\begin{equation*}\label{e3-21}
E_{\lambda}(u_{\lambda})=m_{\lambda}=\sigma(c)+\frac{\lambda c}{2}.
\end{equation*}
Suppose by contradiction that $\|u_{\lambda}\|_{2}^{2}\neq c$. Then for
$\delta:=|\sqrt{c}-\|u_{\lambda}\|_{2}|>0$ and $ M:=c>0,$
we have the continuous path $\kappa: [0,T]\rightarrow H^{1}(\mathbb{R}^{N})$ defined in Proposition \ref{p3-1}. By Proposition \ref{p3-1} $(iii)$, there is $t_{0}\in (0,T)$ such that
\begin{equation*}
\|\kappa(t_{0})\|_{2}^{2}=c\,\text{ and }\,
\|\kappa(t_{0})-u_{\lambda}\|_{2}\geq\delta.
\end{equation*}
Then we have
\begin{eqnarray*}
E_{\lambda} (u_{\lambda})&>&E_{\lambda}(\kappa(t_{0}))\\
&=&E(\kappa(t_{0}))+\frac{\lambda}{2}
\int_{\mathbb{R}^{N}}|\kappa(t_{0})|^{2}dx\\
&\geq&\sigma(c)+\frac{\lambda c}{2},
\end{eqnarray*}
which  is a contradiction. Thus it is further straightforward to see that $E(u_{\lambda})=\sigma(c)$.

\textbf{Finally, we give the proof of Theorem \ref{t3}:} Similar to Lemma \ref{L5-0}, we obtain
\begin{equation*}
E(\psi (0,x))=E(\psi (t,x))\geq \frac{1}{2}A(\psi (t,x))-\frac{1}{4}\left(
\frac{|\mu _{\beta }|}{2}\right) ^{-\frac{\theta }{1-\theta }}c^{3/2}A(\psi
(t,x))^{1/2}.
\end{equation*}%
By the conservation of energy and mass, we have
\begin{equation*}
\Vert \nabla \psi (t,x)\Vert _{2}^{2}\leq C(E,c)\text{ uniformly in }t.
\end{equation*}%
Then there exists a unique global solution for the Cauchy problem to (\ref%
{e1-0}).

Following the
classical arguments of Cazenave and Lions \cite{CL}. Assume that there exist
an $\varepsilon _{0}>0$, a sequence of initial data $\left\{
u_{n}^{0}\right\} $ and a time sequence $\left\{ t_{n}\right\} \subset
\mathbb{R}^{+}$ such that the unique solution $\psi _{n}$ of the Cauchy
problem to (\ref{e1-0}) with initial data $u_{n}^{0}=\psi _{n}(0,\cdot )$
satisfies
\begin{equation*}
\text{dist}_{H^{1}}(u_{n}^{0},\mathcal{M}_{c})<\frac{1}{n}\text{ and dist}%
_{H^{1}}(\psi _{n}(t_{n},\cdot ),\mathcal{M}_{c})\geq \varepsilon _{0}.
\end{equation*}%
Without loss of generality, we may assume that $\left\{ u_{n}^{0}\right\}
\subset S(c)$. Since $\text{dist}_{H^{1}}(u_{n}^{0},\mathcal{M}%
_{c})\rightarrow 0$ as $n\rightarrow \infty $, the conservation laws of the
energy and mass imply that $\{\psi _{n}(t_{n},\cdot )\}$ is a minimizing
sequence for $\sigma (c)$. Then there exists $v_{0}\in \mathcal{M}_{c}$ such
that $\psi _{n}(t_{n},\cdot )\rightarrow v_{0}$ in $H^{1}(\mathbb{R}^{N})$,
which contradicts with dist$_{H^{1}}(\psi _{n}(t_{n},\cdot ),\mathcal{M}%
_{c})\geq \varepsilon _{0}.$ This completes the proof.

\section{Acknowledgments}

J. Sun was supported by the National Natural Science Foundation of China
(Grant No. 12371174) and Shandong Provincial Natural Science Foundation
(Grant No. ZR2020JQ01).

\end{document}